\documentclass[12pt]{article}
\usepackage[utf8]{inputenc}

\usepackage[top=50pt,bottom=60pt,left=78pt,right=76pt]{geometry}
\usepackage{graphicx, amsmath, amsfonts, amsthm, amssymb, latexsym, verbatim,
 pifont, xspace, mathrsfs, enumerate}

\usepackage[colorlinks, hypertexnames=false]{hyperref}
\usepackage[capitalise]{cleveref}

\usepackage{calc}
\usepackage{xcolor}

\usepackage{tikz}
\usetikzlibrary{decorations, arrows, positioning}
\usetikzlibrary{shapes.geometric}

\definecolor{color0}{RGB}{238,20,135}
\definecolor{color1}{RGB}{0,221,164}
\definecolor{color2}{RGB}{86,151,209}
\definecolor{color3}{RGB}{249,185,131}
\definecolor{color4}{RGB}{150,114,196}

\hypersetup{linkcolor=color0, urlcolor=color1, citecolor=color2}

\tikzstyle{vertex}=[circle, fill=black, inner sep=1.5pt]
\tikzstyle{hole}=[circle, draw=black, fill=white, inner sep=1.5pt]

\tikzstyle{lat}=[draw=gray, thick]
\tikzstyle{a}=[-to, draw=black, shorten >=2pt, shorten <=2pt, thick]

\newtheorem{thm}{Theorem}[section]
\newtheorem{prop}[thm]{Proposition}
\newtheorem{cor}[thm]{Corollary}
\newtheorem{lemma}[thm]{Lemma}
\numberwithin{equation}{section}

\renewcommand{\to}{\longrightarrow}
\newcommand{\Z}{\mathbb{Z}}

\newcommand{\J}{\mathscr{J}}

\newcommand{\s}{\mathbf{s}}
\newcommand{\rr}{\mathbf{r}}
\newcommand{\Path}{\operatorname{Path}}

\let\oldmarginpar\marginpar
\renewcommand\marginpar[1]{\-\oldmarginpar[\rrggedleft\footnotesize #1]%
 {\rrggedright\footnotesize #1}}

\usepackage[backend=biber, bibencoding=utf8, giveninits=true]{biblatex}
\bibliography{semimod}

\title{Properties of Congruence Lattices of Graph Inverse Semigroups}
\author{Marina Anagnostopoulou-Merkouri, Zachary Mesyan, \\ and James D. Mitchell}
\date{\today}

\begin{document}

\maketitle

\begin{abstract}
 From any directed graph $E$ one can construct the graph inverse semigroup $G(E)$, whose elements, roughly speaking, correspond to paths in $E$. Wang and Luo showed that the congruence lattice $L(G(E))$ of $G(E)$ is upper-semimodular for every graph $E$, but can fail to be lower-semimodular for some $E$. We provide a simple characterisation of the graphs $E$ for which $L(G(E))$ is lower-semimodular. We also describe those $E$ such that $L(G(E))$ is atomistic, and characterise the minimal generating sets for $L(G(E))$ when $E$ is finite and simple.
 
 \medskip


\end{abstract}

\section{Introduction}

Roughly speaking, a \textit{graph inverse semigroup} $G(E)$ is an inverse semigroup with zero, whose non-zero elements are paths in a (directed) graph $E$, and where the operation is concatenation of those paths or zero, depending on whether one path ends at the vertex where the next path begins. A precise definition is given in the next section. Graph inverse semigroups were introduced by Ash and Hall~\cite{Ash1975aa}, who characterised those graph inverse semigroups that are congruence-free, and showed that every partial order is the partial order of the non-zero $\J$-classes of a graph inverse semigroup. These semigroups also generalise the so-called \textit{polycyclic monoids} of Nivat and Perrot~\cite{Nivat1970aa}, and arise in the study of Leavitt path algebras~\cite{Abrams2017aa} and graph $C^*$-algebras~\cite{Paterson2002aa}.

There has been a number of more recent papers specifically about graph inverse semigroups; see, for example,~\cite{Mesyan2016aa} and references therein. Of particular relevance here are the papers of Wang~\cite{Wang2019aa}, and Luo and Wang~\cite{Luo2021aa}. In~\cite{Wang2019aa}, Wang gives a description of the congruences of a graph inverse semigroup $G(E)$ in terms of certain sets of vertices of the graph $E$ and integer-valued functions on the cycles in $E$. This characterisation is used to show that the lattice $L(G(E))$ of any graph inverse semigroup $G(E)$ is Noetherian, i.e., $G(E)$ does not have any infinite strictly ascending chains of congruences. In~\cite{Luo2021aa}, Luo and Wang show that $L(G(E))$ is always upper-semimodular (\cref{prop-upper-semi}), but not lower-semimodular in general. These results follow a long tradition of studying lattices naturally associated to various algebraic objects. For one example, among many others, it is well-known that the lattice of normal subgroups of a group is modular but not distributive in general. Even if we restrict our attention to the lattices of congruences of semigroups, the literature is rich; see, for example,~\cite{Green1975aa}.

The upper-semimodularity result of Luo and Wang~\cite{Luo2021aa} naturally raises the following question: is it possible to characterise those graphs $E$ for which $L(G(E))$ is lower-semimodular? We answer this question in~\cref{thm-semimod} for arbitrary graphs $E$, and provide a somewhat simpler characterisation for finite simple graphs in~\cref{thm-semimod-mod-dist}. We also completely describe those graphs $E$ such that $L(G(E))$ is atomistic (\cref{thm-atomistic}), and characterise the minimal generating sets for $L(G(E))$  (\cref{thm-generators}), when $E$ is finite and simple. The results in this paper were initially suggested by experiments performed using the \textit{Semigroups} package~\cite{Mitchell2022aa} for GAP~\cite{GAP4}. For example, every connected simple graph $E$ with $4$ vertices, up to isomorphism, such that $L(G(E))$ is lower-semimodular, is shown in~\cref{fig-digraphs-4}.

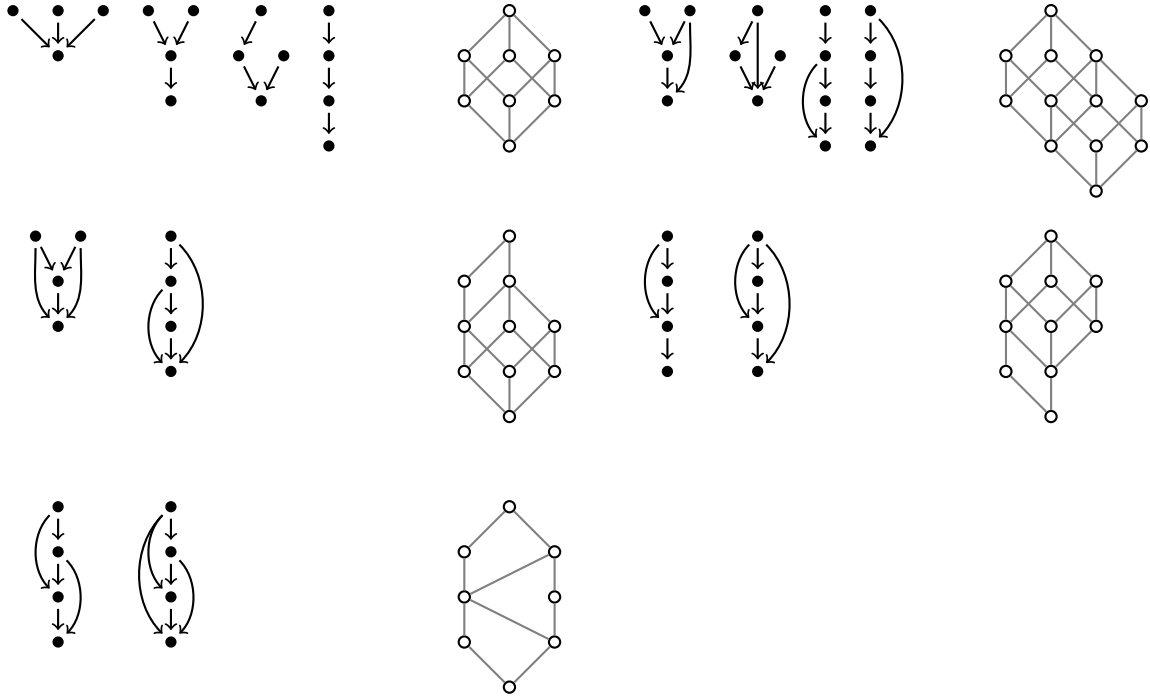
\begin{figure}
 \centering
 \begin{tikzpicture}[scale=0.6]
  \tikzstyle{every node}=[style=vertex]
  \tikzstyle{every path}=[style=a]
  
  \node (1) at (0, 3) {};
  \node (2) at (1, 3) {};
  \node (3) at (2, 3) {};
  \node (4) at (1, 2) {};

  \draw (1) to (4);
  \draw (2) to (4);
  \draw (3) to (4);

  \node (5) at (3, 3) {};
  \node (6) at (4, 3) {};
  \node (7) at (3.5, 2) {};
  \node (8) at (3.5, 1) {};

  \draw (5) to (7);
  \draw (6) to (7);
  \draw (7) to (8);

  \node (9) at (5.5, 3) {};
  \node (10) at (5, 2) {};
  \node (11) at (6, 2) {};
  \node (12) at (5.5, 1) {};

  \draw (9) to (10);
  \draw (10) to (12);
  \draw (11) to (12);

  \node (13) at (7, 3) {};
  \node (14) at (7, 2) {};
  \node (15) at (7, 1) {};
  \node (16) at (7, 0) {};

  \draw (13) to (14);
  \draw (14) to (15);
  \draw (15) to (16);

  \tikzstyle{every node}=[style=hole]
  \tikzstyle{every path}=[style=lat]
  \node[hole] (17) at (11, 3) {};
  \node (18) at (10, 2) {};
  \node (19) at (11, 2) {};
  \node (20) at (12, 2) {};
  \node (21) at (10, 1) {};
  \node (22) at (11, 1) {};
  \node (23) at (12, 1) {};
  \node (24) at (11, 0) {};

  \draw (17) to (18);
  \draw (17) to (19);
  \draw (17) to (20);
  \draw (18) to (21);
  \draw (18) to (22);
  \draw (19) to (21);
  \draw (19) to (23);
  \draw (20) to (22);
  \draw (20) to (23);
  \draw (21) to (24);
  \draw (22) to (24);
  \draw (23) to (24);

  \tikzstyle{every node}=[style=vertex]
  \tikzstyle{every path}=[style=a]
  
  \node (25) at (14, 3) {};
  \node (26) at (15, 3) {};
  \node (27) at (14.5, 2) {};
  \node (28) at (14.5, 1) {};

  \draw (25) to (27);
  \draw (26) to (27);
  \draw (27) to (28);
  \draw[out=-90, in=45] (26) to (28);

  \node (29) at (16.5, 3) {};
  \node (30) at (16, 2) {};
  \node (31) at (17, 2) {};
  \node (32) at (16.5, 1) {};

  \draw (29) to (30);
  \draw (30) to (32);
  \draw (31) to (32);
  \draw (29) to (32);

  \node (33) at (18, 3) {};
  \node (34) at (18, 2) {};
  \node (35) at (18, 1) {};
  \node (36) at (18, 0) {};

  \draw (33) to (34);
  \draw (34) to (35);
  \draw (35) to (36);
  \draw[out=-135, in=135] (34) to (36);

  \node (37) at (19, 3) {};
  \node (38) at (19, 2) {};
  \node (39) at (19, 1) {};
  \node (40) at (19, 0) {};

  \draw (37) to (38);
  \draw (38) to (39);
  \draw (39) to (40);
  \draw[out=-45, in=45] (37) to (40);

  \tikzstyle{every node}=[style=hole]
  \tikzstyle{every path}=[style=lat]
  \node  (43) at (23, 3) {};
  \node  (44) at (22, 2) {};
  \node  (45) at (23, 2) {};
  \node  (46) at (24, 2) {};
  \node  (47) at (22, 1) {};
  \node  (48) at (23, 1) {};
  \node  (49) at (24, 1) {};
  \node  (50) at (25, 1) {};
  \node  (51) at (23, 0) {};
  \node  (52) at (24, 0) {};
  \node  (53) at (25, 0) {};
  \node  (54) at (24, -1) {};

  \draw (43) to (44);
  \draw (43) to (45);
  \draw (43) to (46);
  \draw (44) to (47);
  \draw (44) to (48);
  \draw (45) to (47);
  \draw (45) to (49);
  \draw (46) to (48);
  \draw (46) to (49);
  \draw (46) to (50);
  \draw (47) to (51);
  \draw (48) to (51);
  \draw (48) to (52);
  \draw (49) to (51);
  \draw (49) to (53);
  \draw (50) to (52);
  \draw (50) to (53);
  \draw (51) to (54);
  \draw (52) to (54);
  \draw (53) to (54);

  \tikzstyle{every node}=[style=vertex]
  \tikzstyle{every path}=[style=a]
  
  \node (55) at (0.5, -2) {};
  \node (56) at (1.5, -2) {};
  \node (57) at (1, -3) {};
  \node (58) at (1, -4) {};

  \draw (55) to (57);
  \draw (56) to (57);
  \draw (57) to (58);
  \draw[out=-90, in=45] (56) to (58);
  \draw[out=-90, in=135] (55) to (58);

  \node (59) at (3.5, -2) {};
  \node (60) at (3.5, -3) {};
  \node (61) at (3.5, -4) {};
  \node (62) at (3.5, -5) {};

  \draw (59) to (60);
  \draw (60) to (61);
  \draw (61) to (62);
  \draw[out=-45, in=45] (59) to (62);
  \draw[out=-135, in=135] (60) to (62);

  \tikzstyle{every node}=[style=hole]
  \tikzstyle{every path}=[style=lat]
  
  \node (63) at (11, -2) {};
  \node (64) at (10, -3) {};
  \node (65) at (11, -3) {};
  \node (66) at (10, -4) {};
  \node (67) at (11, -4) {};
  \node (68) at (12, -4) {};
  \node (69) at (10, -5) {};
  \node (70) at (11, -5) {};
  \node (71) at (12, -5) {};
  \node (72) at (11, -6) {};

  \draw (63) to (64);
  \draw (63) to (65);
  \draw (64) to (66);
  \draw (65) to (66);
  \draw (65) to (67);
  \draw (65) to (68);
  \draw (66) to (69);
  \draw (66) to (70);
  \draw (67) to (69);
  \draw (67) to (71);
  \draw (68) to (70);
  \draw (68) to (71);
  \draw (69) to (72);
  \draw (70) to (72);
  \draw (71) to (72);

  \tikzstyle{every node}=[style=vertex]
  \tikzstyle{every path}=[style=a]
  
  \node (73) at (14.5, -2) {};
  \node (74) at (14.5, -3) {};
  \node (75) at (14.5, -4) {};
  \node (76) at (14.5, -5) {};

  \draw (73) to (74);
  \draw (74) to (75);
  \draw (75) to (76);
  \draw[out=-135, in=135] (73) to (75);

  \node (77) at (16.5, -2) {};
  \node (78) at (16.5, -3) {};
  \node (79) at (16.5, -4) {};
  \node (80) at (16.5, -5) {};

  \draw (77) to (78);
  \draw (78) to (79);
  \draw (79) to (80);
  \draw[out=-135, in=135] (77) to (79);
  \draw[out=-45, in=45] (77) to (80);

  \tikzstyle{every node}=[style=hole]
  \tikzstyle{every path}=[style=lat]
  
  \node (81) at (23, -2) {};
  \node (82) at (22, -3) {};
  \node (83) at (23, -3) {};
  \node (84) at (24, -3) {};
  \node (85) at (22, -4) {};
  \node (86) at (23, -4) {};
  \node (87) at (24, -4) {};
  \node (88) at (22, -5) {};
  \node (89) at (23, -5) {};
  \node (90) at (23, -6) {};

  \draw (81) to (82);
  \draw (81) to (83);
  \draw (81) to (84);
  \draw (82) to (85);
  \draw (82) to (86);
  \draw (83) to (85);
  \draw (83) to (87);
  \draw (84) to (86);
  \draw (84) to (87);
  \draw (85) to (88);
  \draw (85) to (89);
  \draw (86) to (89);
  \draw (87) to (89);
  \draw (88) to (90);
  \draw (89) to (90);

  \tikzstyle{every node}=[style=vertex]
  \tikzstyle{every path}=[style=a]
  
  \node (90) at (1, -8) {};
  \node (91) at (1, -9) {};
  \node (92) at (1, -10) {};
  \node (93) at (1, -11) {};

  \draw (90) to (91);
  \draw (91) to (92);
  \draw (92) to (93);
  \draw[out=-135, in=135] (90) to (92);
  \draw[out=-45, in=45] (91) to (93);

  \node (94) at (3.5, -8) {};
  \node (95) at (3.5, -9) {};
  \node (96) at (3.5, -10) {};
  \node (97) at (3.5, -11) {};

  \draw (94) to (95);
  \draw (95) to (96);
  \draw (96) to (97);
  \draw[out=-135, in=135] (94) to (96);
  \draw[out=-135, in=135] (94) to (97);
  \draw[out=-45, in=45] (95) to (97);

  \tikzstyle{every node}=[style=hole]
  \tikzstyle{every path}=[style=lat]
  
  \node (98) at (11, -8) {};
  \node (99) at (10, -9) {};
  \node (100) at (12, -9) {};
  \node (101) at (10, -10) {};
  \node (102) at (12, -10) {};
  \node (103) at (10, -11) {};
  \node (104) at (12, -11) {};
  \node (105) at (11, -12) {};

  \draw (98) to (99);
  \draw (98) to (100);
  \draw (99) to (101);
  \draw (100) to (101);
  \draw (100) to (102);
  \draw (101) to (103);
  \draw (101) to (104);
  \draw (102) to (104);
  \draw (103) to (105);
  \draw (104) to (105);

 \end{tikzpicture}
\caption{Every connected simple graph $E$ with $4$ vertices, such that $L(G(E))$ is lower-semimodular, together with the corresponding lattice $L(G(E))$.} \label{fig-digraphs-4}
\end{figure}

\section{Definitions and statements of main results} \label{section-defs}

\subsection{Graphs}

A \textit{(directed) graph} $E = (E^0, E^1, \textbf{s}, \textbf{r})$ is a quadruple consisting of two sets, $E^0$ and $E^1$, and two functions $\textbf{s}, \textbf{r}: E^1 \to E^0$, called \textit{source} and \textit{range}, respectively. The elements of $E^0$ and $E^1$ are referred to as \textit{vertices} and \textit{edges}, respectively. A vertex $v \in E^0$ satisfying $\s^{-1}(v) = \varnothing$ is called a \textit{sink}. A sequence $p=e_1e_2\cdots e_n$ of (not necessarily distinct) edges $e_i \in E^1$, such that $\textbf{r}(e_i) = \textbf{s}(e_{i + 1})$ for $1\leq i\leq n - 1$, is a \textit{path} from $\s(e_1)$ to $\rr(e_n)$. Here we set $\s(p) = \s(e_1)$ and $\rr(p) = \rr(e_n)$, and refer to $n$ as the \textit{length} of $p$. We view the elements of $E^0$ as paths of length $0$, and denote by $\Path(E)$ the set of all paths in $E$. A path $p = e_1\cdots e_n$ where $n\geq 1$, $\s(p) = \rr(p)$, and $\s(e_i) \neq \s(e_j)$ for all $i\neq j$, is a \textit{cycle}. Two distinct edges $e, f\in E^1$, such that $\s(e) = \s(f)$ and $\rr(e) = \rr(f)$, are called \textit{parallel}. A graph containing no cycles is called \textit{acyclic}, while an acyclic graph without parallel edges is called \textit{simple}. A graph $E$ is \textit{finite} if both $E^0$ and $E^1$ are finite.

Given a graph $E$ and vertices $u,v\in E^0$, we write $u > v$ if $u\neq v$ and there is a path $p \in \Path(E)$ such that $\s(p)=u$ and $\rr(p)=v$. It is easy to see that $\geq$, defined in the obvious way from $>$, is a preorder on $E^0$. Next, let $H$ be a subset of $E^0$. Then $H$ is \emph{downward directed}  if it is non-empty, and for all $u,v \in H$ there exists $w \in H$ such that $u \geq w$ and $v\geq w$. We say that $H$ is \textit{hereditary} if $u\geq v$ implies that $v\in H$, for all $u \in H$ and $v\in E^0$. Supposing that $H$ is non-empty, $H$ is called a \emph{strongly connected component} if $u\geq v$ for all $u,v \in H$, and $H$ is maximal with respect to this property. If $H$ is non-empty and hereditary, then being a strongly connected component amounts to satisfying $u\geq v$ for all $u,v \in H$. Finally, we denote by $E\setminus H$ the subgraph of $E$ induced by $H$. Specifically, $(E\setminus H)^0 = E^0 \setminus H$, 
\[
(E\setminus H)^1 = E^1 \setminus \{e \in E^1 \mid \s(e) \in H \text{ or } \rr(e) \in H\},
\] 
and the source and range functions, $\s_{E\setminus H}$ and $\rr_{E\setminus H}$, are the restrictions of $\s$ and $\rr$, respectively, to $(E\setminus H)^1$.

\subsection{Inverse semigroups}

Let $S$ be a semigroup, i.e., a set with an associative binary operation. We say that $S$ is an \textit{inverse semigroup}, if for every $x\in S$ there exists a unique $x^{-1}\in S$ satisfying $xx^{-1}x = x$ and $x^{-1}xx^{-1} = x^{-1}$. 

Given a graph $E$, we define the \textit{graph inverse semigroup $G(E)$ of $E$} to be the inverse semigroup with zero, generated by $E^0$ and $E^1$, together with a set of elements $E^{-1} = \{e^{-1} \mid e\in E^1\}$, that satisfies the following four axioms, for all $u, v\in E_0$ and $e, f\in E^1$:
\begin{description}
 \item[(V)] $vu = \delta_{v, u}v$,
 \item[(E1)] $\textbf{s}(e)e = e\textbf{r}(e) = e$,
 \item[(E2)] $\textbf{r}(e)e^{-1} = e^{-1}\textbf{s}(e) = e^{-1}$,
 \item[(CK1)] $e^{-1}f = \delta_{e, f}\textbf{r}(e)$.
\end{description}
The symbol $\delta$ appearing in (V) and (CK1) is the Kronecker delta. For every $v\in E^0$ we define $v^{-1} = v$, and for every $q = e_1\cdots e_n\in \Path(E)$ we define $q^{-1} = e_n^{-1}\cdots e_1^{-1}$. It follows directly from the above axioms that every non-zero element in $G(E)$ can be written in the form $pq^{-1}$ for some $p, q\in \Path(E)$. It is routine to show that $G(E)$ is an inverse semigroup, with $(pq^{-1})^{-1} = qp^{-1}$ for every non-zero $pq^{-1}\in G(E)$. Moreover, $G(E)$ is finite if and only if $E$ is finite and acyclic.

If $S$ is any semigroup and $\rho\subseteq S\times S$ is an equivalence relation, then $\rho$ is called a \textit{congruence} if $(zx, zy), (xz, yz)\in \rho$, for all $(x, y)\in \rho$ and all $z\in S$. The \textit{diagonal congruence} is $\Delta_S = \{(x, x) \mid x\in S\}$, and the \textit{universal congruence} is $S\times S$. If $R$ is any subset of $S \times S$, then we denote by $R^{\sharp}$ the least congruence on $S$ containing $R$; this is called the congruence \textit{generated by $R$}.

\subsection{Lattices}

A partially ordered set $(L, \leq)$ is a \textit{lattice}, if for all $a, b\in L$ there exists an infimum $a \wedge b$, called the \textit{meet} of $a$ and $b$, and a supremum $a\vee b$, called the \textit{join} of $a$ and $b$. When the order is clear from the context we will write $L$ instead of $(L, \leq)$. For instance, if $X$ is any set, then the power set $\mathcal{P}(X)$ of $X$ forms a lattice under containment, where $\wedge$ is intersection and $\vee$ is union. Similarly, the collection of all congruences $L(S)$ on a semigroup $S$ forms a lattice, with respect to containment, where $\rho \vee \sigma = (\rho \cup \sigma)^{\sharp}$ and $\rho \wedge \sigma = \rho \cap \sigma$, for all $\rho, \sigma \in L(S)$. By convention, the diagonal congruence $\Delta_S$ on $S$ is the join of the empty set of congruences. A lattice $L$ is \textit{complete} if every subset $K\subseteq L$ has an infimum $\bigwedge K$ and a supremum $\bigvee K$. Examples of complete lattices include all finite lattices, the power set lattice $\mathcal{P}(X)$ of any set $X$, and the lattice of congruences $L(S)$ of any semigroup $S$. 

Two lattices $L_1$ and $L_2$ are \textit{order-isomorphic} if there exists a bijection $\Psi : L_1 \to L_2$, such that $\Psi(a\vee b) = \Psi(a) \vee \Psi(b)$ and $\Psi(a\wedge b) = \Psi(a) \wedge \Psi(b)$, for all $a, b\in L_1$. A subset $L'$ of a lattice $L$ is called a \textit{sublattice} of $L$ if it forms a lattice under the same join and meet operations as $L$. We say that $L$ is \textit{generated} by a subset $X$ if every element of $L$ can be expressed as a join of finitely many elements of $X$. In this situation, the elements of $X$ are called \textit{generators} of $L$. For $a, b\in L$, we say that $b$ \textit{covers} $a$, and write $a \prec b$, if $a < b$ and there is no element $c \in L$, such that $a < c < b$. If $L$ is a lattice with a least element $0$, then $a\in L$ is an \textit{atom} in case $0\prec a$. A lattice is \textit{atomistic} if it can be generated using only atoms. 

Let $L$ be a lattice. Then $L$ is \textit{modular} if $a\leq c$ implies that $(a \vee b) \wedge c = a\vee (b \wedge c)$, for all $a, b, c\in L$. Moreover, $L$ is \textit{upper-semimodular} if $a\wedge b \prec a, b$ implies that $a, b\prec a\vee b$, for all $a, b\in L$. Likewise, $L$ is \textit{lower-semimodular} if $a, b \prec a\vee b $ implies that $a\wedge b \prec a, b$, for all $a, b\in L$. Finally, $L$ is \textit{distributive} if $a\wedge(b\vee c) = (a\wedge b)\vee (a\wedge c)$ and $a\vee(b\wedge c) = (a\vee b) \wedge (a\vee c)$, for all $a, b, c\in L$.

It is well-known that every distributive lattice is modular. Moreover, every modular lattice is both upper- and lower-semimodular, and the converse also holds for finite lattices. A lattice $L$ is distributive if and only if neither the pentagon $\mathfrak{N}_5$ nor the diamond $\mathfrak{M}_3$, shown in \cref{fig-pentagon}, is a sublattice of $L$. Similarly, a lattice $L$ is modular if and only if the pentagon $\mathfrak{N}_5$ is not a sublattice of $L$.  See, for example, \cite{Grtzer1978aa} for further details.

\begin{figure}
 \centering
 \begin{tikzpicture}
  \tikzstyle{every node}=[style=hole]
  \tikzstyle{every path}=[style=lat]
  \node (0) at (1, 0) {};
  \node (a) at (-0.5, 2) {};
  \node (b) at (1, 2) {};
  \node (c) at (2.5, 2) {};
  \node (1) at (1, 4) {};

  \draw (1) to (a);
  \draw (1) to (b);
  \draw (1) to (c);
  \draw (a) to (0);
  \draw (b) to (0);
  \draw (c) to (0);
 \end{tikzpicture}
 \qquad\qquad
 \begin{tikzpicture}
  \tikzstyle{every node}=[style=hole]
  \tikzstyle{every path}=[style=lat]
  \node (1) at (1, 0) {};
  \node (2) at (2, 2) {};
  \node (3) at (1, 4)   {};
  \node (4) at (0, 3)   {};
  \node (5) at (0, 1)   {};

  \draw[lat] (1) to (2);
  \draw[lat] (2) to (3);
  \draw (3) to (4);
  \draw (4) to (5);
  \draw (5) to (1);
 \end{tikzpicture}
 \caption{The diamond lattice $\mathfrak{M}_3$ and the pentagon lattice $\mathfrak{N}_5$.} \label{fig-pentagon}
\end{figure}
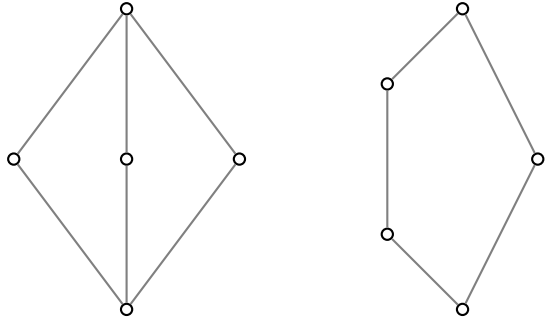

\subsection{Congruence lattices on graph inverse semigroups}

Let $E$ be a graph. Given a subset $H$ of $E^0$, we denote by $C(H)$ the set of cycles $c=e_1\dots e_n \in \Path(E)$ such that $\s(e_i) \in H$ for each $i$. As defined in~\cite{Luo2021aa}, a \emph{Wang triple} $(H,W,f)$ on $E$ consists of a hereditary set $H \subseteq E^0$, a set 
\[
W \subseteq \{v \in E^0 \setminus H \mid |\s_{E\setminus H}^{-1}(v)| = 1\},
\] 
and a \emph{cycle} function $f: C(E^0) \to \Z^+\cup\{\infty\}$ (i.e., $f(c) = 1$ for all $c \in C(H)$, $f(c) = \infty$ for all $c \notin C(H\cup W)$, and the restriction of $f$ to $C(W)$ is invariant under cyclic permutations). In~\cite{Luo2021aa}, the term ``congruence triple" is used for this concept.

Given a Wang triple $(H, W, f)$ on a graph $E$, we define $\varrho(H, W, f)$ to be the corresponding congruence, generated by the following set:
\begin{eqnarray*}
 (H\times \{0\}) \cup \{(w, ee^{-1}) \mid w \in W,\ \mathbf{s}(e) = w,\ \mathbf{r}(e)\not\in H\} \\ \cup \{(c^{f(c)}, \s(c)) \mid c\in C(W), \ f(c) \in \Z^+\}.
\end{eqnarray*}
Also, given two Wang triples $(H_1,W_1,f_1)$ and $(H_2,W_2,f_2)$ on $E$, write $(H_1,W_1,f_1) \leq (H_2,W_2,f_2)$ if $H_1 \subseteq H_2$, $W_1\setminus H_2 \subseteq W_2$, and $f_2(c) \mid f_1(c)$ for all $c \in C(E^0)$. (Here $\mid$ denotes ``divides", and it is understood that $\infty \mid \infty$, and $n \mid \infty$ for all $n \in \Z^+$.) According to~\cite[Corollary 1.1]{Luo2021aa} and~\cite[Lemma 2.18]{Wang2019aa}, $\leq$ is a partial order on the set of all Wang triples on a graph. Moreover, Luo and Wang characterise the lattice of congruences on a graph inverse semigroup, according to the associated Wang triples. We will repeatedly require this characterisation, and so we state it in the next proposition.

\begin{prop}[Proposition 1.2 in \cite{Luo2021aa}] \label{Luo-Wang-Thrm}
 Let $E$ be a graph, let $G(E)$ be the graph inverse semigroup of $E$, and let $L(G(E))$ be the lattice of congruences on $G(E)$. The function $(H, W, f) \mapsto \varrho(H, W, f)$ is an order-isomorphism between the set of all Wang triples on $E$, ordered by $\leq$, and $L(G(E))$, ordered by containment.
\end{prop}

In light of this proposition, we will abuse notation by identifying Wang triples with the corresponding congruences, writing $(H, W, f)$ instead of $\varrho(H, W, f)$, and $(H_1,W_1,f_1) \subseteq (H_2,W_2,f_2)$ instead of $(H_1,W_1,f_1) \leq (H_2,W_2,f_2)$. If $E$ is acyclic, then the component $f$ in a Wang triple $(H, W, f)$ is redundant, and we will write $(H, W, \varnothing)$ instead.

Next, we record another frequently used result, mentioned in the introduction.

\begin{prop}[Theorem 1.3 in \cite{Luo2021aa}]\label{prop-upper-semi}
Let $E$ be a graph, let $G(E)$ be the graph inverse semigroup of $E$, and let $L(G(E))$ be the lattice of congruences on $G(E)$. Then $L(G(E))$ is upper-semimodular.
\end{prop}

In~\cite[Example 2]{Luo2021aa} Luo and Wang produce a graph $E$ such that $L(G(E))$ is not lower-semimodular. For another example, if $E$ is the graph in \cref{fig-not-yet-drawn}, then it can be shown using the \textit{Semigroups} package~\cite{Mitchell2022aa} for GAP~\cite{GAP4}, that $L(G(E))$ is isomorphic to the lattice in \cref{fig-not-yet-drawn}, which is easily seen to not be lower-semimodular.

\begin{figure}
 \centering
 \begin{tikzpicture}
  \node [vertex] (1) at (1, 2.4) {};
  \node [vertex] (2) at (1, 1.1) {};
  \node [vertex] (3) at (0, 0)   {};
  \node [vertex] (4) at (2, 0)   {};

  \draw[style=a] (1) to (2);
  \draw[style=a] (2) to (3);
  \draw[style=a] (2) to (4);
 \end{tikzpicture}
 \qquad
 \begin{tikzpicture}
  \tikzstyle{every node}=[style=hole]
  \tikzstyle{every path}=[style=lat]
  \node [fill=color2] (8) at (2, 4) {};

  \node [fill=color3] (10) at (0.5, 3) {};
  \node (13) at (1.5, 3) {};
  \node (5) at (2.5, 3) {};
  \node [fill=color3] (11) at (3.5, 3) {};

  \node (12) at (0, 2) {};
  \node (6) at (1, 2) {};
  \node (3) at (2, 2) {};
  \node (14) at (3, 2) {};
  \node (7) at (4, 2) {};

  \node (2) at (1, 1) {};
  \node [fill=color4] (9) at (2, 1) {};
  \node (4) at (3, 1) {};

  \node (1) at (2, 0) {};

  \draw (8) to (10);
  \draw (8) to (13);
  \draw (8) to (5);
  \draw (8) to (11);

  \draw (10) to (12);
  \draw (10) to (6);
  \draw (13) to (12);
  \draw (13) to (3);
  \draw (13) to (14);
  \draw (5) to (6);
  \draw (5) to (3);
  \draw (5) to (7);
  \draw (11) to (14);
  \draw (11) to (7);

  \draw (12) to (2);
  \draw (12) to (9);
  \draw (6) to (2);
  \draw (3) to (2);
  \draw (3) to (4);
  \draw (14) to (9);
  \draw (14) to (4);
  \draw (7) to (4);

  \draw (2) to (1);
  \draw (9) to (1);
  \draw (4) to (1);

 \end{tikzpicture}
 \caption{A graph $E$, together with $L(G(E))$, which is not lower-semimodular. The vertices of $L(G(E))$ shown in orange are covered by their join, shown in blue, but they do not cover their meet, shown in purple.} \label{fig-not-yet-drawn}
\end{figure}
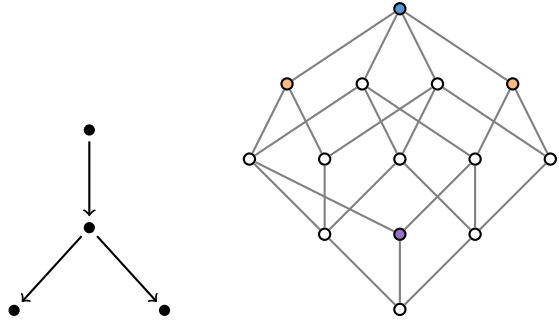

\subsection{Main results}

To state our first main theorem, which characterises those graphs $E$ such that the congruence lattice $L(G(E))$ of the corresponding graph inverse semigroup $G(E)$ is lower-semimodular, we require the following definition. Let $E$ be a graph and $v \in E^0$. We refer to $v$ as a \textit{forked} vertex, if there exist distinct edges $e,f \in \s^{-1}(v)$ such that the following properties hold:
\begin{enumerate}[\rm (i)]
 \item $\rr(g) \not\geq \rr(e)$ for all $g \in \s^{-1}(v) \setminus \{e\}$;
 \item $\rr(g) \not\geq \rr(f)$ for all $g \in \s^{-1}(v) \setminus \{f\}$.
\end{enumerate}

\begin{thm}\label{thm-semimod}
Let $E$ be a graph, let $G(E)$ be the graph inverse semigroup of $E$, and let $L(G(E))$ be the lattice of congruences on $G(E)$. Then $E$ has no forked vertices if and only if $L(G(E))$ is lower-semimodular.
\end{thm}

\cref{thm-semimod} has a somewhat simpler form when restricted to finite simple graphs, in which case additional characterisations can be given.

\begin{cor}\label{thm-semimod-mod-dist}
Let $E$ be a finite acyclic graph, let $G(E)$ be the graph inverse semigroup of $E$, and let $L(G(E))$ be the lattice of congruences of $G(E)$. Then the following are equivalent:
 \begin{enumerate}[\rm (i)]
  \item $L(G(E))$ is lower-semimodular;
  \item $L(G(E))$ is modular;
  \item $L(G(E))$ is distributive.
 \end{enumerate}
If $E$ is simple, then these conditions are also equivalent to the following:
 \begin{enumerate}[\rm (iv)]
  \item $\rr(e) \geq \rr(f)$ or $\rr(f) \geq \rr(e)$ for all $e, f\in E^1$ such that $\s(e) = \s(f)$.
 \end{enumerate}
\end{cor}

In the next of our main theorems, we characterise those graphs $E$ such that $L(G(E))$ is atomistic.

\begin{thm}\label{thm-atomistic}
Let $E$ be a graph, let $G(E)$ be the graph inverse semigroup of $E$, and let $L(G(E))$ be the lattice of congruences of $G(E)$. Then every congruence in $L(G(E))$ is the join of a (possibly infinite) collection of atoms if and only if for every $v\in E^0$ one of the following holds:
 \begin{enumerate}[\rm (i)]
  \item $|\s^{-1}(v)| = 0$;
  \item $|\s^{-1}(v)| = 1$, $v$ does not belong to a cycle, and $v > u$ for some $u \in E^0$ such that $|\s^{-1}(u)| \neq 1$;
  \item $|\s^{-1}(v)| \geq 2$, and $\rr(e) \geq v$ for all $e \in \s^{-1}(v)$.
 \end{enumerate}
Moreover, $L(G(E))$ is atomistic if and only if, in addition to the above conditions on all vertices, $E^0$ has only finitely many strongly connected components and vertices $v$ such that $|\s^{-1}(v)| = 1$.
\end{thm}

An example of a graph satisfying the conditions of \cref{thm-atomistic} is given in \cref{fig-atomistic}.

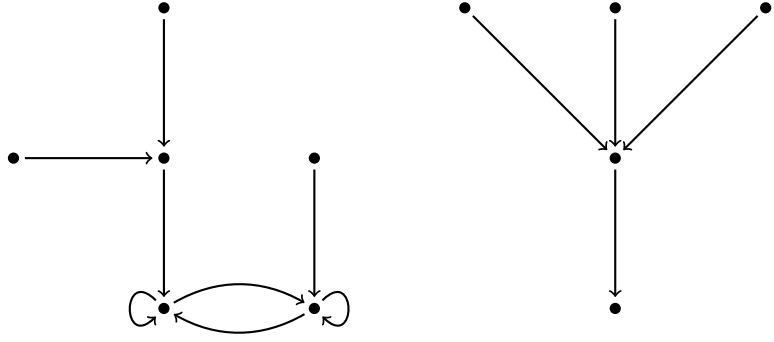
\begin{figure}
 \centering
 \begin{tikzpicture}
  \tikzstyle{every node}=[style=vertex]
  \node  (1) at (2, 4) {};
  \node  (2) at (6, 4) {};
  \node  (3) at (8, 4) {};
  \node  (4) at (10, 4) {};

  \node  (5) at (0, 2) {};
  \node  (6) at (2, 2) {};
  \node  (7) at (4, 2) {};
  \node  (8) at (8, 2) {};

  \node (9) at (2, 0) {};
  \node (10) at (4, 0) {};
  \node (11) at (8, 0) {};

  \draw[style=a] (1) to (6);
  \draw[style=a] (2) to (8);
  \draw[style=a] (3) to (8);
  \draw[style=a] (4) to (8);
  \draw[style=a] (5) to (6);
  \draw[style=a] (6) to (9);
  \draw[style=a] (7) to (10);
  \draw[style=a] (8) to (11);
  \draw[style=a, bend left] (9) to (10);
  \draw[style=a, loop, out=135, in=225, looseness=16] (9) to (9);
  \draw[style=a, bend left] (10) to (9);
  \draw[style=a, loop, in=-45, out=45, looseness=16] (10) to (10);

 \end{tikzpicture}
 \caption{An example of a graph satisfying the conditions of
 \cref{thm-atomistic}.} \label{fig-atomistic}
\end{figure}

The conditions of \cref{thm-atomistic} simplify significantly when the graph is finite and acyclic.

\begin{cor}\label{cor-atomistic}
Let $E$ be a finite acyclic graph, let $G(E)$ be the graph inverse semigroup of $E$, and let $L(G(E))$ be the lattice of congruences of $G(E)$. Then the following are equivalent:
 \begin{enumerate}[\rm(i)]
  \item $|s^{-1}(v)| \leq 1$ for all $v\in E ^ 0$;
  \item $L(G(E))$ is isomorphic to the power set lattice $\mathcal{P}(E^0)$;
  \item $L(G(E))$ is atomistic.
 \end{enumerate}
\end{cor}

For graphs $E$ such that $G(E)$ is infinite, $L(G(E))$ being isomorphic to $\mathcal{P}(E^0)$ is generally not equivalent to $L(G(E))$ being atomistic. For example, if $|E^0| = \aleph_0$, then the number of atoms in $L(G(E))$ is at most $\aleph_0$ (this follows from \cref{atoms} and \cref{Luo-Wang-Thrm}), and so the cardinality of the lattice generated by atoms is at most $\aleph_0$ also. On the other hand, $|\mathcal{P}(E^0)| = 2^{\aleph_0} > \aleph_0$. Hence $L(G(E))$ is not atomistic if it is isomorphic to $\mathcal{P}(E^0)$.

The last of our main theorems establishes a generating set for $L(G(E))$ in terms of the graph $E$, when it is finite and simple.

\begin{thm}\label{thm-generators}
Let $E$ be a finite simple graph, let $G(E)$ be the graph inverse semigroup of $E$, let $L(G(E))$ be the lattice of congruences on $G(E)$, and let $\mathcal{A} \subseteq L(G(E))$. Then  $\mathcal{A}$ generates $L(G(E))$ if and only if $\mathcal{A}$ contains all  the congruences of the following types:
 \begin{enumerate}[\rm (i)]
  \item $(\{v\}, \varnothing, \varnothing)$, such that $v\in E^0$ and $|\s^{-1}(v)| = 0$;
  \item $(H, \{v\}, \varnothing)$, such that $v\in E^0$, $|\s^{-1}(v)| > 0$, and $H$ is a minimal (with respect containment) hereditary subset of $E^0$ satisfying $|\s_{E\setminus H}^{-1}(v)| = 1$.
 \end{enumerate}
\end{thm}

The statement in \cref{thm-generators} does not hold for graphs with parallel edges. For example, if $E$ is the graph given in \cref{fig-zak-counter-ex}, then the only congruences on $G(E)$ of types (i) and (ii) in \cref{thm-generators} are of the form $(\{v\}, \varnothing, \varnothing)$, where $v\in E^0$ and $|\s^{-1}(v)| = 0$. It follows (using \cref{join-meet}) that the congruence $(E^0, \varnothing, \varnothing)$ on $G(E)$ is not a join of congruences of types (i) and (ii).

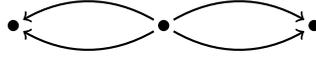
\begin{figure}
 \centering
 \begin{tikzpicture}
  \node [vertex] (1)   at (0, 0) {};
  \node [vertex] (2)   at (2, 0) {};
  \node [vertex] (3)   at (4, 0) {};

  \draw[style=a, bend left]  (2) to (1);
  \draw[style=a, bend right] (2) to (1);
  \draw[style=a, bend left]  (2) to (3);
  \draw[style=a, bend right] (2) to (3);
 \end{tikzpicture}
 \caption{A graph $E$ with parallel edges, for which the conclusion of
 \cref{thm-generators} does not hold.}
 \label{fig-zak-counter-ex}
\end{figure}

\section{Meets, joins, and covers}

In this section, we describe the circumstances under which one congruence covers another in a graph inverse semigroup, in terms of Wang triples. This fact will be used repeatedly in the paper.

We begin by stating a result from~\cite{Luo2021aa} that describes the meets and joins of Wang triples, for convenience of reference.

\begin{prop}[Lemmas 2.7 and 2.8 in~\cite{Luo2021aa}] \label{join-meet}
Let $E$ be a graph, let $(H_1,W_1,f_1)$ and $(H_2,W_2,f_2)$ be Wang triples on $E$, and set 
 \[
 V_0 = \{v \in (W_1 \cup W_2) \setminus (H_1 \cup H_2) \mid \s_{E\setminus (H_1 \cup H_2)}^{-1}(v) = \varnothing\}
 \] 
 and
 \begin{align*}
  J = \{v \in (W_1 \cup W_2) \setminus (H_1 \cup H_2) & \mid \exists e_1\cdots e_n \in \Path(E) \ \forall i\in \{2, \dots, n\} \\
    & (\s(e_1) = v, \, \rr(e_n) \in V_0, \, \s(e_i) \in W_1 \cup W_2)\}.
 \end{align*}
Then 
 \[
 (H_1,W_1,f_1) \land (H_2,W_2,f_2) = (H_1 \cap H_2, (W_1 \cap H_2) \cup (W_2 \cap H_1) \cup ((W_1 \cap W_2)\setminus V_0), \mathrm{lcm}(f_1,f_2)),
 \] 
where $\mathrm{lcm}(f_1,f_2)(c) = \mathrm{lcm}(f_1(c),f_2(c))$ for all $c \in C(E^0)$, and 
 \[
 (H_1,W_1,f_1) \lor (H_2,W_2,f_2) = (H_1 \cup H_2 \cup J, (W_1 \cup W_2)\setminus (H_1 \cup H_2 \cup J), \mathrm{gcd}(f_1,f_2)),
 \] 
where $\mathrm{gcd}(f_1,f_2)(c) = \mathrm{gcd}(f_1(c),f_2(c))$ for all $c \in C(E^0)$.
\end{prop}

Next, we characterise the situations where one Wang triple covers another. 

\begin{prop} \label{min-triple}
Let $E$ be a graph, and let $(H_1,W_1,f_1)$ and $(H_2,W_2,f_2)$ be Wang triples on $E$, such that $(H_1,W_1,f_1) \subsetneq (H_2,W_2,f_2)$. Then $(H_1,W_1,f_1) \prec (H_2,W_2,f_2)$, i.e., $(H_2,W_2,f_2)$ covers $(H_1,W_1,f_1)$, if and only if one of the following holds:
\begin{enumerate}[\rm (i)]
  \item $H_1 = H_2$, $W_1 = W_2$, and $f_2 \prec f_1$. (I.e., there is a cycle $c \in C(W_1)$ such that $f_1(c)/f_2(c)$ is a prime integer, and $f_1(d)=f_2(d)$ for all $d \in C(W_1)\setminus \{c\}$.)
  \item $H_1 = H_2$, $|W_2 \setminus W_1| = 1$, and $f_1 = f_2$.
  \item $H_1 \subsetneq H_2$, $W_1 \setminus H_2 = W_2$, 
  \[
  W_1 \cap H_2 = \{v \in H_2\setminus H_1 \mid |\s^{-1}_{E\setminus H_1}(v)|=1\},
  \] 
  $f_1(c) = f_2(c)$ for all $c \in C(W_1)$, and for each hereditary set $H_1 \subsetneq H' \subsetneq H_2$ there exists $v \in W_1\setminus H'$ such that $\rr(e) \in H'$ for all $e \in \s^{-1}(v)$.
 \end{enumerate}
Moreover, if (iii) holds, then $H_2\setminus H_1$ is downward directed.
\end{prop}

\begin{proof}
If (i) holds, then $(H_1,W_1,f_1) \prec (H_2,W_2,f_2)$, by~\cite[Lemma 2.4]{Luo2021aa}. If (ii) holds, then it follows immediately from the definition of the ordering on Wang triples (or~\cite[Lemma 2.3]{Luo2021aa}) that $(H_1,W_1,f_1) \prec (H_2,W_2,f_2)$. Let us now suppose that (iii) holds and that 
 \[
 (H_1,W_1,f_1) \subseteq (H',W',f') \subseteq (H_2,W_2,f_2)
 \] 
for some Wang triple $(H',W',f')$. We will show that either $(H',W',f') = (H_1,W_1,f_1)$ or $(H',W',f') = (H_2,W_2,f_2)$. Notice that necessarily $f_1(c) = f'(c) = f_2(c)$ for all $c \in C(W_1)$.

Suppose that $H_1 \subsetneq H' \subsetneq H_2$. Then, by hypothesis, there exists $v \in W_1\setminus H'$ such that $\rr(e) \in H'$ for all $e \in \s^{-1}(v)$. Thus $v \in W_1 \setminus (H' \cup W')$, which contradicts $(H_1,W_1,f_1) \subseteq (H',W',f')$. It follows that either $H' = H_1$ or $H' = H_2$. In the first case, $H' = H_1$, 
 \[
 W_1 = W_1 \setminus H_1 = W_1 \setminus H' \subseteq W',
 \] 
which implies that $W_1 \setminus H_2 = W' \setminus H_2 = W_2$. Since, by hypothesis, $W_1$ contains all $v \in H_2 \setminus H_1$ such that $|\s^{-1}_{E\setminus H_1}(v)|=1$, we see that $W_1 = W'$, and so $(H',W',f') = (H_1,W_1,f_1)$. In the second case, $H' = H_2$,
 \[
 W' = W' \setminus H' = W' \setminus H_2 \subseteq W_2 = W_1 \setminus H_2 = W_1 \setminus H' \subseteq W',
 \] 
which implies that $W' = W_2$. Since $W_2 \subseteq W_1$, it follows that $f' = f_2$, and so $(H',W',f') = (H_2,W_2,f_2)$, as desired.

For the converse, suppose that $(H_1,W_1,f_1) \prec (H_2,W_2,f_2)$. Let us also suppose, for the moment, that $H_1 = H_2$. If $W_1 = W_2$, then $f_2 \prec f_1$, by~\cite[Lemma 2.4]{Luo2021aa}. If $W_1 \subsetneq W_2$, then $|W_2 \setminus W_1| = 1$, and $f_1 = f_2$, by~\cite[Lemma 2.3]{Luo2021aa}. Thus, if $H_1 = H_2$, then exactly one of (i) or (ii) must hold. Let us now assume that $H_1 \subsetneq H_2$. Then $W_1 \setminus H_2 = W_2$, 
 \[
 W_1 \cap H_2 = \{v \in H_2\setminus H_1 \mid |\s^{-1}_{E\setminus H_1}(v)|=1\},
 \] 
and $f_1(c) = f_2(c)$ for all $c \in C(W_1)$, by~\cite[Lemma 2.1]{Luo2021aa}. Therefore to conclude the proof of the main claim, it suffices to take a hereditary set $H_1 \subsetneq H' \subsetneq H_2$ and show that there exists $v \in W_1\setminus H'$ such that $\rr(e) \in H'$ for all $e \in \s^{-1}(v)$.

Suppose, on the contrary, that $|\s^{-1}_{E\setminus H'}(v)| = 1$ for all $v \in W_1\setminus H'$. Let $W' = W_1\setminus H'$, and let $f': C(E^0) \to \mathbb{Z}^+ \cup \{\infty\}$ be the cycle function such that $f_1(c) = f'(c)$ for all $c\in C(W')$ and $f'(c)=1$ for all $c\in C(H')$. Then $(H',W', f')$ is a Wang triple such that 
 \[
 (H_1,W_1,f_1) \subsetneq (H',W',f') \subsetneq (H_2,W_2,f_2),
 \] 
contradicting our hypothesis. Therefore there must exist $v \in W_1\setminus H'$ such that $\rr(e) \in H'$ for all $e \in \s^{-1}(v)$.

For the final claim, suppose that (iii) holds, and let $u,v \in H_2\setminus H_1$. Now suppose that for all $w \in H_2\setminus H_1$ either $u \not\geq w$ or $v \not\geq w$. Let $G_1 = H_1 \cup \{x \in E^0 \mid u \geq x\}$, and for each $i > 1$ let 
 \[
 G_i = G_{i-1} \cup \{x \in W_1 \setminus G_{i-1} \mid \rr(\s^{-1}(x)) \subseteq G_{i-1}\}.
 \] 
Then, clearly, $H' = \bigcup_{i=1}^{\infty}G_i$ is hereditary, $H_1 \subsetneq H' \subseteq H_2$, and there is no $x \in W_1\setminus H'$ such that $\rr(e) \in H'$ for all $e \in \s^{-1}(x)$. Condition (iii) then implies that $H' = H_2$. Notice also that, by hypothesis, $v \notin G_1$, and that if $v \in W_1$, then it cannot be the case that $\rr(\s^{-1}(v)) \subseteq G_{i-1}$ for some $i > 1$. It follows that $v \notin H'$, in contradiction to $H' = H_2$. Therefore there must exist $w \in H_2\setminus H_1$ such that $u \geq w$ and $v \geq w$, i.e., $H_2\setminus H_1$ is downward directed.
\end{proof}

If the graph $E$ is finite and acyclic, then the conditions in \cref{min-triple} can be simplified substantially, as the next corollary shows. \cref{lemma-set-diff} follows fairly quickly from \cref{min-triple}, but the proof is omitted because we will not use this result directly.

\begin{cor}\label{lemma-set-diff}
Let $E$ be a finite acyclic graph, and let $(H_1, W_1, \varnothing)$ and $(H_2, W_2, \varnothing)$ be Wang triples on $E$, such that $(H_1, W_1, \varnothing) \subseteq (H_2, W_2, \varnothing)$. Then $(H_1, W_1, \varnothing) \prec (H_2, W_2, \varnothing)$ if and only if $|(H_2 \cup W_2) \setminus (H_1 \cup W_1)| = 1$.
\end{cor}

\section{Modularity}

In this section we will prove \cref{thm-semimod} and \cref{thm-semimod-mod-dist}. We begin with a sequence of lemmas that will culminate in the proof of \cref{thm-semimod}.

\begin{lemma} \label{bifurcation}
Let $E$ be a graph containing a forked vertex. Then $L(G(E))$ is not lower-semimodular.
\end{lemma}
\begin{proof}
By hypothesis, there exist $v \in E^0$ and distinct $e,f \in \s^{-1}(v)$, such that $\rr(g) \not\geq \rr(e)$ for all $g \in \s^{-1}(v) \setminus \{e\}$, and $\rr(g) \not\geq \rr(f)$ for all $g \in \s^{-1}(v) \setminus \{f\}$. Let $u=\rr(e)$, $w=\rr(f)$, $X = \{x \in E^0 \mid v \geq x\}$, $H_u = \{x \in X \mid x \not\geq u\}$, and $H_w = \{x \in X \mid x \not\geq w\}$. Then, clearly, $X$, $H_u$, and $H_w$ are hereditary. Next, let
\[
W_u = \{y \in X\setminus H_u \mid |\s^{-1}_{E\setminus H_u}(y)|=1\} \text{ and } W_w = \{y \in X\setminus H_w \mid |\s^{-1}_{E\setminus H_w}(y)|=1\}.
\] 
Also, for each set $H \subseteq E^0$ let us denote by $f_H: C(E^0) \to \mathbb{Z}^+ \cup \{\infty\}$ the function such that $f_H(c) = 1$ for all $c \in C(H)$ and $f_H(c) = \infty$ for all $c \in C(E^0 \setminus H)$. Then $(H_u, W_u, f_u)$ and $(H_w, W_w, f_w)$ are Wang triples, where $f_u = f_{H_u \cup W_u}$ and $f_w = f_{H_w \cup W_w}$. Also, by construction, $v \in W_u$ and $v \in W_w$. By \cref{join-meet},
\[
(H_u, W_u, f_u) \lor (H_w, W_w, f_w) = (X, \varnothing, f_X),
\] 
since $|\s^{-1}_{E\setminus (H_u \cup H_w)}(v)|=0$ and $v \in (W_u \cup W_w) \setminus (H_u \cup H_w)$. Using \cref{join-meet} again, since $v \in V_0$, 
\[
(H_u, W_u, f_u) \land (H_w, W_w, f_w) = (H_u \cap H_w, W, f_{uw}),
\] 
for some set $W$ such that $W\setminus H_u \subseteq W_u \setminus \{v\}$, and $f_{uw} = \mathrm{lcm}(f_u, f_v)$. Then 
\[
(H_u \cap H_w, W, f_{uw}) \subsetneq (H_u, W_u \setminus \{v\},  f_{H_u \cup (W_u \setminus \{v\})}) \subsetneq (H_u, W_u, f_u),
\]
since $w \in H_u \setminus H_w$ implies that $H_u \cap H_w \subsetneq H_u$. Therefore, to conclude that $L(G(E))$ is not lower-semimodular it suffices to show that $(H_u, W_u, f_u) \prec (X, \varnothing, f_X)$ and $(H_w, W_w, f_w) \prec (X, \varnothing, f_X)$. Given the symmetry of the situation, we shall only show that $(H_w, W_w, f_w) \prec (X, \varnothing, f_X)$. By \cref{min-triple}, it is enough to prove that for any hereditary set $H_w \subsetneq H' \subsetneq X$ there exists $y \in W_w \setminus H'$ such that $\rr(g) \in H'$ for all $g \in \s^{-1}(y)$.

Suppose that $H_w \subsetneq H' \subsetneq X$ for some hereditary set $H'$, and let $x \in H' \setminus H_w$. Then $x \geq w$, and so $w \in H'$. Hence, by construction, $\rr(g) \in H'$ for all $g \in \s^{-1}(v)$. Moreover $v \in W_w \setminus H'$, since $H' \neq X$, giving the desired conclusion.
\end{proof}

\begin{lemma} \label{no-fork}
Let $E$ be a graph with no forked vertices. Then for any pair $(H_1,W_1,f_1)$ and $(H_2,W_2,f_2)$ of Wang triples  on $E$, we have $V_0 \cap W_1 \cap W_2 = \varnothing$, where $V_0$ is the set defined in \cref{join-meet}, and 
 \[
 (H_1,W_1,f_1) \land (H_2,W_2,f_2) = (H_1 \cap H_2, (W_1 \cap H_2) \cup (W_2 \cap H_1) \cup (W_1 \cap W_2), \mathrm{lcm}(f_1,f_2)).
 \]
\end{lemma}

\begin{proof}
Suppose that there exists $v \in V_0 \cap W_1 \cap W_2$. Since $v \in V_0$, we have $\s_{E\setminus (H_1 \cup H_2)}^{-1}(v) = \varnothing$. Since $v \in W_1 \cap W_2$, there must exist (distinct) $e,f \in \s^{-1}(v)$ such that $\rr(e) \in H_1\setminus H_2$, $\rr(f) \in H_2\setminus H_1$, and $\rr(g) \in H_1 \cap H_2$ for all $g \in \s^{-1}(v)\setminus \{e,f\}$. Since $H_1$ and $H_2$ are hereditary, it follows that $\rr(g)\not\geq\rr(e)$ and $\rr(g)\not\geq\rr(f)$ for all $g \in \s^{-1}(v)\setminus \{e,f\}$, $\rr(e) \not\geq \rr(f)$, and $\rr(f) \not\geq \rr(e)$. That is, $v \in E^0$ is forked.

Thus if $E$ has no forked vertices, then $V_0 \cap W_1 \cap W_2 = \varnothing$. The claim about $(H_1,W_1,f_1) \land (H_2,W_2,f_2)$ now follows from \cref{join-meet}.
\end{proof}

\begin{lemma} \label{J-empty}
Let $E$ be a graph, and suppose that $(H_1,W_1,f_1)$ and $(H_2,W_2,f_2)$ are Wang triples on $E$, such that the set $J$ defined in \cref{join-meet} is empty, and 
\[
(H_1,W_1,f_1) \prec (H_1,W_1,f_1) \lor (H_2,W_2,f_2).
\] 
Then 
\[
(H_1,W_1,f_1) \land (H_2,W_2,f_2) \prec (H_2,W_2,f_2).
\]
\end{lemma}

\begin{proof}
By \cref{join-meet}, 
 \[
 (H_1,W_1,f_1) \lor (H_2,W_2,f_2) = (H, (W_1 \cup W_2)\setminus H, \mathrm{gcd}(f_1,f_2)),
 \] 
where $H =  H_1 \cup H_2 \cup J = H_1 \cup H_2$. Since 
 \[
 (H_1,W_1,f_1) \prec (H_1,W_1,f_1) \lor (H_2,W_2,f_2),
 \] 
by \cref{min-triple}, there are three possible cases, which we examine individually.

\textit{Case 1:} $H_1 = H$, $W_1 = (W_1 \cup W_2)\setminus H$, and $\mathrm{gcd}(f_1,f_2) \prec f_1$. Then $H_2 \subseteq H_1$, $W_2 \setminus H_1 \subseteq W_1$, and $f_2 \prec \mathrm{lcm}(f_1,f_2)$. Given that $V_0 \subseteq J = \varnothing$ and $W_1 \cap H_1 = \varnothing$, it follows that
 \[
 (W_1 \cap H_2) \cup (W_2 \cap H_1) \cup ((W_1 \cap W_2)\setminus V_0) = \varnothing \cup (W_2 \cap H_1) \cup (W_1 \cap W_2) = W_2.
 \] 
Therefore, by \cref{join-meet} and \cref{min-triple}, 
 \[
 (H_1,W_1,f_1) \land (H_2,W_2,f_2) = (H_2,W_2,\mathrm{lcm}(f_1,f_2)) \prec (H_2,W_2,f_2).
 \]

\textit{Case 2:} $H_1 = H$, $|((W_1 \cup W_2)\setminus H) \setminus W_1|=1$, and $f_1 = \mathrm{gcd}(f_1,f_2)$. Then $H_2 \subseteq H_1$, $f_2 = \mathrm{lcm}(f_1,f_2)$, and
 \begin{eqnarray*}
 |W_2 \setminus ((W_1 \cap H_2) \cup (W_2 \cap H_1) \cup (W_1 \cap W_2))| = |W_2 \setminus (H_1 \cup W_1)|\\ = |(W_1 \cup W_2)\setminus (H_1 \cup H_2 \cup W_1)| = |((W_1 \cup W_2)\setminus H) \setminus W_1|=1.
 \end{eqnarray*}
Therefore, again using the fact that $V_0 = \varnothing$, by \cref{join-meet} and \cref{min-triple}, 
 \[
 (H_1,W_1,f_1) \land (H_2,W_2,f_2) = (H_2,(W_1 \cap H_2) \cup (W_2 \cap H_1) \cup (W_1 \cap W_2),f_2) \prec (H_2,W_2,f_2).
 \]

\textit{Case 3:} $H_1 \subsetneq H$, $W_1 \setminus H = (W_1 \cup W_2)\setminus H$, 
 \[
 W_1 \cap H = \{v \in H\setminus H_1 \mid |\s^{-1}_{E\setminus H_1}(v)|=1\},
 \] 
$f_1(c) = \mathrm{gcd}(f_1,f_2)(c)$ for all $c \in C(W_1)$, and for each hereditary set $H_1 \subsetneq H' \subsetneq H$ there exists $v \in W_1\setminus H'$ such that $\rr(e) \in H'$ for all $e \in \s^{-1}(v)$. Then $W_1 \setminus H = (W_1 \cup W_2)\setminus H$ implies that $W_2 \subseteq H_1 \cup W_1$. Moreover, $f_2(c) = \mathrm{lcm}(f_1,f_2)(c)$ for all $c \in C(W_1)$, from which it follows that $f_2(c) = \mathrm{lcm}(f_1,f_2)(c)$ for all $c \in C(W_1 \cup W_2)$, since $f_1(c) = 1$ for all $c\in C(H_1)$. Notice also that given a hereditary set $H_1 \subsetneq H' \subsetneq H$ and $v \in W_1\setminus H'$ such that $\rr(e) \in H'$ for all $e \in \s^{-1}(v)$, it must be the case that $v \in  W_1 \cap H_2$, since otherwise $v \in V_0 \subseteq J$. Therefore, by \cref{join-meet}, 
 \[
 (H_1,W_1,f_1) \land (H_2,W_2,f_2) = (H_1 \cap H_2, (W_1 \cap H_2) \cup W_2, \mathrm{lcm}(f_1,f_2)).
 \]

Now, since $H_1 \subsetneq H_1 \cup H_2$, we have $H_1 \cap H_2 \subsetneq H_2$. Also $((W_1 \cap H_2) \cup W_2)\setminus H_2 = W_2$, and
 \begin{eqnarray*}
 ((W_1 \cap H_2) \cup W_2) \cap H_2 = W_1 \cap H_2 = W_1 \cap H = \{v \in H_2\setminus H_1 \mid |\s^{-1}_{E\setminus H_1}(v)|=1\}\\  = \{v \in H_2\setminus (H_1\cap H_2) \mid |\s^{-1}_{E\setminus (H_1\cap H_2)}(v)|=1\},
 \end{eqnarray*}
since $\rr(\s^{-1}(v)) \subseteq H_2$ for any $v \in H_2$. Thus, by \cref{min-triple}, to conclude that 
 \[
 (H_1,W_1,f_1) \land (H_2,W_2,f_2) \prec (H_2,W_2,f_2)
 \] 
it suffices to check that for each hereditary set $H_1 \cap H_2 \subsetneq H' \subsetneq H_2$ there exists $v \in ((W_1 \cap H_2) \cup W_2)\setminus H'$ such that $\rr(e) \in H'$ for all $e \in \s^{-1}(v)$. Given such a hereditary set $H'$, the set $H_1 \cup H'$ is also hereditary, and $H_1 \subsetneq H_1 \cup H' \subsetneq H$. Hence, as noted above, there exists $v \in (W_1 \cap H_2)\setminus (H_1 \cup H')$ such that $\rr(e) \in H_1 \cup H'$ for all $e \in \s^{-1}(v)$. That is, $v \in (W_1 \cap H_2)\setminus H'$. Since $v \in H_2$, we see that 
 \[
 \rr(e) \in (H_1 \cup H') \cap H_2 = (H_1 \cap H_2) \cup (H' \cap H_2) = (H_1 \cap H_2) \cup H' = H'
 \]
for all $e \in \s^{-1}(v)$, as desired.
\end{proof}

\begin{lemma} \label{containment}
Let $E$ be a graph, and suppose that $(H_1,W_1,f_1)$ and $(H_2,W_2,f_2)$ are Wang triples on $E$, such that $H_2 \subseteq H_1$, $J \neq \varnothing$, and
\[
(H_1,W_1,f_1) \prec (H_1,W_1,f_1) \lor (H_2,W_2,f_2).
\] 
Then $V_0 = W_2 \setminus (H_1 \cup W_1)$ and $|V_0|=1$. (See \cref{join-meet} for the definitions of $J$ and $V_0$.)
\end{lemma}

\begin{proof}
By \cref{join-meet}, 
 \[
 (H_1,W_1,f_1) \lor (H_2,W_2,f_2) = (H, (W_1 \cup W_2)\setminus H, \mathrm{gcd}(f_1,f_2)),
 \] 
where $H =  H_1 \cup H_2 \cup J$. Since 
 \[
 (H_1,W_1,f_1) \prec (H_1,W_1,f_1) \lor (H_2,W_2,f_2),
 \] 
by \cref{min-triple}, there are three possible cases. However, the hypothesis that $J \neq \varnothing$ implies that $H_1 \neq H$, which rules out two of those cases. Thus $H_1 \subsetneq H$, $W_1 \setminus H = (W_1 \cup W_2)\setminus H$, and $f_1(c) = \mathrm{gcd}(f_1,f_2)(c)$ for all $c \in C(W_1)$, among other conditions. Since $H_2 \subseteq H_1$ and $W_1 \cap H_1 = \varnothing$, we have $$W_2 \setminus (H_1\cup J) = W_2 \setminus H \subseteq W_1 \setminus H = W_1 \setminus J,$$ which implies that $W_2 \setminus (H_1 \cup W_1) \subseteq J$. We begin by showing that $V_0 = W_2 \setminus (H_1 \cup W_1)$.

Note that since $H_2 \subseteq H_1$, for all $v \in W_1$ we have $1 = |\s^{-1}_{E\setminus H_1}(v)| = |\s^{-1}_{E\setminus (H_1\cup H_2)}(v)|$, and so $v \notin V_0$. Thus $V_0 \cap W_1 = \varnothing$. Hence $V_0 \subseteq W_2 \setminus (H_1 \cup H_2) = W_2 \setminus H_1$, and therefore $V_0 \subseteq W_2 \setminus (H_1 \cup W_1)$.

Now, suppose that $v \in (W_2 \setminus (H_1 \cup W_1)) \setminus V_0$. Then, in particular $v \in J \setminus V_0$, and so $|\s^{-1}_{E\setminus H_1}(v)| = |\s^{-1}_{E\setminus H_2}(v)| = 1$, but $v \notin W_1$. Therefore,
 \begin{eqnarray*}
  (H_1,W_1,f_1) \subsetneq (H_1,W_1 \cup \{v\},f_1) \subsetneq (H_1 \cup J,(W_1 \cup W_2)\setminus H,\mathrm{gcd}(f_1,f_2))\\
  = (H_1,W_1,f_1) \lor (H_2,W_2,f_2),
 \end{eqnarray*}
contrary to hypothesis. Thus $W_2 \setminus (H_1 \cup W_1) \subseteq V_0$, and so $V_0 = W_2 \setminus (H_1 \cup W_1)$.

It remains to show that $|V_0|=1$. Since $J \neq \varnothing$, there exists $v \in V_0 = W_2 \setminus (H_1 \cup W_1)$. Let
 \begin{eqnarray*}
 J_v = \{u \in J \mid \exists e_1\cdots e_n \in \Path(E) \ \forall i\in \{2, \dots, n\} \ \\ (\s(e_1) = u, \, \rr(e_n) = v, \, \s(e_i) \in W_1 \cup W_2)\}.
 \end{eqnarray*}
Then it is easy to see that $H_1 \cup J_v$ is a hereditary set, and that $|\s^{-1}_{E\setminus (H_1 \cup J_v)}(w)| = 1$ for all $w \in W_1 \setminus J_v$. Thus $(H_1 \cup J_v,W_1 \setminus J_v,f)$ is a well-defined Wang triple, where $f(c) = f_1(c)$ for all $c\in W_1 \setminus J_v$, $f(c) = 1$ for all $c \in C(H_1 \cup J_v)$, and $f(c) = \infty$ for all $c \notin C(H_1 \cup J_v \cup W_1)$. Now, since $v \in J_v$, and hence $J_v \neq \varnothing$, we have
 \[
 (H_1,W_1,f_1) \subsetneq (H_1 \cup J_v,W_1 \setminus J_v, f) \subseteq (H, (W_1 \cup W_2)\setminus H, \mathrm{gcd}(f_1,f_2)),
 \] 
which implies that $H_1 \cup J_v = H = H_1 \cup J$, and so $J = J_v$ (since $J\cap H_1 = \varnothing$). It follows from the definition of $J$ that $v$ is the unique element of $V_0$, and hence $|V_0|=1$.
\end{proof}

\begin{lemma} \label{no-containment}
Let $E$ be a graph, and suppose that $(H_1,W_1,f_1)$ and $(H_2,W_2,f_2)$ are Wang triples on $E$, such that $H_2 \not\subseteq H_1$, $J\neq \varnothing$ (see \cref{join-meet}), and 
 \[
 (H_1,W_1,f_1) \prec (H_1,W_1,f_1) \lor (H_2,W_2,f_2).
 \] 
Then $J \subseteq W_1$.
\end{lemma}

\begin{proof}
By \cref{join-meet}, 
 \[
 (H_1,W_1,f_1) \lor (H_2,W_2,f_2) = (H_1 \cup H_2 \cup J, (W_1 \cup W_2)\setminus (H_1 \cup H_2 \cup J), \mathrm{gcd}(f_1,f_2)).
 \]
Now suppose that $v \in J \setminus W_1$, let
\begin{eqnarray*}
 J_v = \{u \in J \mid \exists e_1\cdots e_n \in \Path(E) \ \forall i\in \{2, \dots, n\} \ \\ (\s(e_1) = u, \, \rr(e_n) = v, \, \s(e_i) \in W_1 \cup W_2)\},
\end{eqnarray*}
 and let $H = H_1 \cup H_2 \cup (J \setminus J_v)$. We note that for all $w \in J \setminus J_v$, either $\rr(e) \in H_1 \cup H_2$ for all $e \in \s^{-1}(w)$, or there is a unique $e \in \s^{-1}(w)$ such that $\rr(e) \notin H_1\cup H_2$, in which case $\rr(e) \in J \setminus J_v$. It follows that $H$ is a hereditary set.

Next, suppose that $w \in W_1 \setminus H$. We claim that $|\s^{-1}_{E\setminus H}(w)|=1$. If $w \notin J_v$, then $w \in W_1 \setminus (H_1 \cup H_2 \cup J)$, and so $|\s^{-1}_{E\setminus (H_1 \cup H_2 \cup J)}(w)|=1$, by~\cite[Lemma 2.8]{Luo2021aa} (\cref{join-meet}). Since $|\s^{-1}_{E\setminus H}(w)| \leq 1$ for all $w \in W_1 \cup W_2$, in this case it follows that $|\s^{-1}_{E\setminus H}(w)| = 1$. Therefore we may suppose that $w \in J_v$. Then $w \neq v$, since $w \in W_1$ and $v \in J \setminus W_1$, from which it follows that $w \notin V_0$ (since $V_0 \cap J_v \subseteq \{v\}$). Hence there is a unique $e \in \s^{-1}(w)$ such that $\rr(e) \in J_v$, and so once again $|\s^{-1}_{E\setminus H}(w)| = 1$.

Since $H$ is hereditary and $|\s^{-1}_{E\setminus H}(w)|=1$ for all $w \in W_1 \setminus H$, we conclude that $(H, W_1 \setminus H, f)$ is a well-defined Wang triple, where $f(c) = f_1(c)$ for all $c\in W_1 \setminus H$, $f(c) = 1$ for all $c \in C(H)$, and $f(c) = \infty$ for all $c \notin C(H \cup W_1)$. Since $H_2 \not\subseteq H_1$ and $J_v \neq \varnothing$, it follows that
 \[
 (H_1,W_1,f_1) \subsetneq (H, W_1 \setminus H, f)\subsetneq (H_1 \cup H_2 \cup J, (W_1 \cup W_2)\setminus (H_1 \cup H_2 \cup J), \mathrm{gcd}(f_1,f_2)),
 \] 
contradicting the hypothesis that 
 \[
 (H_1,W_1,f_1) \prec (H_1,W_1,f_1) \lor (H_2,W_2,f_2).
 \] 
Therefore $J \setminus W_1 = \varnothing$.
\end{proof}

\begin{proof}[Proof of \cref{thm-semimod}]
($\Longleftarrow$): It follows immediately from \cref{bifurcation} that if $L(G(E))$ is lower-semimodular, then $E$ has no forked vertices.

\medskip

\noindent ($\Longrightarrow$): Suppose that $E$ has no forked vertices. It suffices to take Wang triples $(H_1,W_1,f_1)$ and $(H_2,W_2,f_2)$ such that
 \[
 (H_1,W_1,f_1) \prec (H_1,W_1,f_1) \lor (H_2,W_2,f_2),
 \] and show that
 \[
 (H_1,W_1,f_1) \land (H_2,W_2,f_2) \prec (H_2,W_2,f_2).
 \]

By \cref{J-empty}, we may assume that $J \neq \varnothing$. By \cref{join-meet}, 
 \[
 (H_1,W_1,f_1) \lor (H_2,W_2,f_2) = (H, (W_1 \cup W_2)\setminus H, \mathrm{gcd}(f_1,f_2)),
 \] 
where $H =  H_1 \cup H_2 \cup J$. Since $J \neq \varnothing$, it cannot be the case that $H_1 = H$. Hence, by \cref{min-triple}, the following conditions hold:
 \begin{enumerate}[(a)]
  \item $H_1 \subsetneq H$,
  \item $W_1 \setminus H = (W_1 \cup W_2)\setminus H$,
  \item $W_1 \cap H = \{v \in H \setminus H_1 \mid |\s^{-1}_{E\setminus H_1}(v)|=1\},$
  \item $f_1(c) = \mathrm{gcd}(f_1,f_2)(c)$ for all $c \in C(W_1)$,
  \item for each hereditary set $H_1 \subsetneq H' \subsetneq H$ there exists $v \in W_1\setminus H'$ such that $\rr(e) \in H'$ for all $e \in \s^{-1}(v)$.
 \end{enumerate}
Since $E$ has no forked vertices, by \cref{no-fork}, 
\[
(H_1,W_1,f_1) \land (H_2,W_2,f_2) = (H_1 \cap H_2, W, \mathrm{lcm}(f_1,f_2)),
\] 
where 
 \[
 W = (W_1 \cap H_2) \cup (W_2 \cap H_1) \cup (W_1 \cap W_2).
 \] 
Moreover, by (d), $\mathrm{lcm}(f_1,f_2)(c) = f_2(c)$ for all $c \in C(W_1)$. Since $f_1(c) = 1$ for all $c \in H_1$, it follows that $\mathrm{lcm}(f_1,f_2)(c) = f_2(c)$ for all $c \in C(W)$.

Let us next consider the case where $H_2 \subseteq H_1$, i.e., where $H_1 \cap H_2 = H_2$. Then, by \cref{containment}, $V_0 = W_2 \setminus (H_1 \cup W_1)$ and $|V_0|=1$. Thus, using the fact that $W_2 \cap H_2 = \varnothing$, gives
 \[
 W_2 \setminus W = W_2 \setminus (W_2 \cap (H_1 \cup W_1)) = W_2 \setminus (H_1 \cup W_1) = V_0,
 \] 
and so $|W_2 \setminus W| = 1$. Since no vertex in $V_0$ can belong to a cycle, and $H_2 \subseteq H_1$, it also follows that $\mathrm{lcm}(f_1,f_2)(c) = f_2(c)$ for all 
\[
c \in C(H_2 \cup (W_2\setminus W) \cup W) = C(H_2 \cup W_2),
\] 
and so $\mathrm{lcm}(f_1,f_2) = f_2$. Therefore, by \cref{min-triple}, 
 \[
 (H_1,W_1,f_1) \land (H_2,W_2,f_2)= (H_2,W,f_2) \prec (H_2,W_2,f_2),
 \] 
as desired.

Hence we may assume that $H_2 \not\subseteq H_1$, i.e., that $H_1 \cap H_2 \subsetneq H_2$. Then $J \subseteq W_1$, by \cref{no-containment}. Since, by (b), $W_2 \setminus H \subseteq W_1 \setminus H$, it follows that 
\[
W_2 \subseteq H_1 \cup W_1 \cup J = H_1 \cup W_1.
\] 
Therefore 
 \[
 W\setminus H_2 = (W_2 \cap (H_1 \cup W_1))\setminus H_2 = W_2 \cap (H_1 \cup W_1) = W_2.
 \]
Next, using (c) and $J \subseteq W_1$, gives 
 \[
 (W_1 \cap H_2) \cup J = W_1 \cap (H_2 \cup J) = W_1 \cap H = \{v \in (H_2 \setminus H_1) \cup J \mid |\s^{-1}_{E\setminus H_1}(v)|=1\},
 \] 
from which it follows that 
 \[
 W \cap H_2 = W_1 \cap H_2 = (W_1 \cap H) \setminus J = \{v \in H_2 \setminus (H_1 \cap H_2) \mid |\s^{-1}_{E\setminus (H_1 \cap H_2)}(v)|=1\},
 \] 
since $\rr(\s^{-1}(v)) \subseteq H_2$ for any $v \in H_2$. Therefore, recalling that 
 \[
 (H_1,W_1,f_1) \land (H_2,W_2,f_2) = (H_1 \cap H_2, W, \mathrm{lcm}(f_1,f_2))
 \] 
and that $\mathrm{lcm}(f_1,f_2)(c) = f_2(c)$ for all $c \in C(W)$, by \cref{min-triple}, to conclude that
 \[
 (H_1,W_1,f_1) \land (H_2,W_2,f_2) \prec (H_2,W_2,f_2),
 \] 
it suffices to show that for each hereditary set $H_1 \cap H_2 \subsetneq H' \subsetneq H_2$ there exists $v \in W\setminus H'$ such that $\rr(e) \in H'$ for all $e \in \s^{-1}(v)$.

Let $H_1 \cap H_2 \subsetneq H' \subsetneq H_2$ be a hereditary set. Then $H_1 \cup H'$ is a hereditary set, and 
\[
H_1 \subsetneq H_1 \cup H' \subsetneq H_1 \cup H_2 \subseteq H.
\] 
Hence, by (e), there exists $v \in W_1\setminus (H_1 \cup H') = W_1\setminus H'$ such that $\rr(e) \in H_1 \cup H'$ for all $e \in \s^{-1}(v)$. If $v \in H_2$ for some such $v$, then $v \in (W_1 \cap H_2)\setminus H' \subseteq W\setminus H'$, and $\rr(e) \in (H_1 \cup H') \cap H_2 = H'$ for all $e \in \s^{-1}(v)$, as required. Thus to conclude the proof, it is enough to show that it cannot be the case that $v \notin H_2$ for all $v \in W_1\setminus H'$ satisfying $\rr(\s^{-1}(v)) \subseteq H_1 \cup H'$.

Assume that $v \notin H_2$ for all $v \in W_1\setminus H'$ satisfying $\rr(\s^{-1}(v)) \subseteq H_1 \cup H'$. Then, in particular, $v \in V_0$ for all such vertices. Let 
\[
V' = \{u \in V_0 \mid \rr(\s^{-1}(u)) \subseteq H_1 \cup H'\},
\] 
and let
 \begin{eqnarray*}
 J' = \{u \in J \mid \exists e_1\cdots e_n \in \Path(E) \ \forall i\in \{2, \dots, n\} \ \\ (\s(e_1) = u, \, \rr(e_n) \in V', \, \s(e_i) \in W_1 \cup W_2)\}.
 \end{eqnarray*}
Then, using the fact that $J \subseteq W_1$, we see that $K = H_1 \cup H' \cup J'$ is hereditary. Also, $H_1 \subsetneq K \subsetneq H = H_1 \cup H_2 \cup J$, since $J \cap (H_1\cup H_2) = \varnothing$.

We claim that $|\s^{-1}_{E\setminus K}(u)|=1$ for all $u \in W_1 \setminus K$. Let  $u \in W_1 \setminus K$, and let $e \in \s^{-1}(u)$ be the unique edge such that $\rr(e) \notin H_1$. Since $W_1 \setminus K \subseteq W_1\setminus H'$, if $\rr(e) \in H' \, (\subseteq H_1 \cup H')$, then, by assumption, $u \notin H_2$, and so $u \in V'$, contradicting $u \notin J'$. Thus $\rr(e) \notin H'$. Moreover, $\rr(e) \notin J'$, since $u \notin J'$, and therefore $\rr(e) \notin K$. Since $|\s^{-1}_{E\setminus H_1}(u)|=1$, it follows that $|\s^{-1}_{E\setminus K}(u)|=1$. 
 
Therefore, defining $f(c) = f_1(c)$ for all $c\in W_1 \setminus K$, $f(c) = 1$ for all $c \in C(K)$, and $f(c) = \infty$ for all $c \notin C(K \cup W_1)$, we have
 \[
 (H_1,W_1,f_1) \subsetneq (K, W_1\setminus K,f)\subsetneq (H, (W_1 \cup W_2)\setminus H, \mathrm{gcd}(f_1,f_2)),
 \] 
contrary to hypothesis. (The existence of $K$ with the above properties also contradicts (e).) Hence it cannot be the case that $v \notin H_2$ for all $v \in W_1\setminus H'$ satisfying $\rr(\s^{-1}(v)) \subseteq H_1 \cup H'$, as required.
\end{proof}

We need one more lemma to prove \cref{thm-semimod-mod-dist}.

\begin{lemma}\label{lemma-union}
Let $E$ be a finite acyclic graph. If $(H_1, W_1, \varnothing)$ and $(H_2, W_2, \varnothing)$ are Wang triples on $E$ such that $H_1\cup W_1 = H_2 \cup W_2$, then $H_1 = H_2$ and $W_1 = W_2$.
\end{lemma}

\begin{proof}
Seeking a contradiction, suppose that there exists $v\in H_1 \cap W_2$. Since $v\in W_2$, there exists $e_0\in \s^{-1}(v)$ such that $\rr(e_0) \not \in H_2$. Since $v\in H_1$ and $H_1$ is hereditary, $\rr(e_0)\in H_1$ also. Since $H_1\cup W_1 = H_2 \cup W_2$ and $\rr(e_0)\not\in H_2$, it follows that $\rr(e_0)\in W_2$. Repeating this construction gives $e_1 \in \s^{-1}(\rr(e_0))$ such that $\rr(e_1) \in H_1 \cap W_2$, and so on. Since $E$ is finite and acyclic, this process must yield a path $e_0e_1 \cdots e_n$ where $\rr(e_i)\in H_1 \cap W_2$ for all $i$, and where $\rr(e_n)$ is a sink. But then $|\s_{E\setminus H_2}^{-1}(\rr(e_n))| = 0$, contradicting  $\rr(e_n) \in W_2$. Therefore $H_1\cap W_2 = \varnothing$, and so $H_1\subseteq H_2$. By symmetry, $H_2 \cap W_1 = \varnothing$, and so $H_2 \subseteq H_1$, which implies that $H_1 = H_2$. Since $H_1 \cap W_1 = \varnothing = H_2 \cap W_2$, it follows that $W_1 = W_2$.
\end{proof}

\begin{proof}[Proof of \cref{thm-semimod-mod-dist}]
\noindent (i) $\Longrightarrow$ (ii): Suppose that $L(G(E))$ is lower-semimodular. By \cref{prop-upper-semi}, $L(G(E))$ is also upper-semimodular. Since $L(G(E))$ is finite, it follows that $L(G(E))$ is modular (see \cite[IV.2, Corollary 3]{Grtzer1978aa}).
 
\medskip

\noindent (ii) $\Longrightarrow$ (iii): Suppose that $L(G(E))$ is modular. Then the pentagon lattice $\mathfrak{N}_5$ (see \cref{fig-pentagon}) is not a sublattice of $L(G(E))$, as discussed above. Therefore to show that $L(G(E))$ is distributive it suffices to prove that the diamond lattice $\mathfrak{M_3}$ (also shown in \cref{fig-pentagon}) is not a sublattice of $L(G(E))$.

Seeking a contradiction, suppose that there exist distinct Wang triples $(H_1, W_1, \varnothing)$, $(H_2, W_2, \varnothing)$, and $(H_3, W_3, \varnothing)$ on $E$, such that the joins and meets of any two are equal. In this case, none of these three is contained in any of the others. We denote $(H_1, W_1, \varnothing) \vee (H_2, W_2, \varnothing)$ by $(H^{\vee}, W^{\vee}, \varnothing)$, and $(H_1, W_1, \varnothing)\wedge (H_2, W_2, \varnothing)$ by $(H^{\wedge}, W^{\wedge}, \varnothing)$.

By \cref{join-meet},
 \begin{equation}\label{eq-number-0}
 H^{\vee} \cup W^{\vee} = H_i \cup W_i \cup H_j \cup W_j = H_i\cup W_i \cup [(H_j \cup W_j) \setminus (H_i\cup W_i)]
 \end{equation}
for all distinct $i, j\in \{1, 2, 3\}$. Therefore
 \begin{equation}\label{eq-number-1}
  (H_j \cup W_j) \setminus (H_i \cup W_i) = (H_k \cup W_k) \setminus (H_i \cup W_i),
 \end{equation}
whenever $\{i, j,k\} = \{1, 2, 3\}$. Using \cref{join-meet} again,
 \[
 (H^{\wedge}, W^{\wedge}, \varnothing) = (H_1\cap H_2, (W_1\cap H_2) \cup (W_2\cap H_1) \cup ((W_1 \cap W_2) \setminus V_0), \varnothing).
 \]
Since $L(G(E))$ is modular, and therefore lower-semimodular, by \cref{thm-semimod}, $E$ has no forked vertices. Hence, by \cref{no-fork}, $W_1 \cap W_2 \cap V_0 = \varnothing$, and so
 \[
 H^{\wedge}\cup W^{\wedge} = (H_1\cap H_2)\cup (W_1\cap H_2) \cup (W_2\cap H_1) \cup(W_1 \cap W_2) = (H_1\cup W_1) \cap (H_2\cup W_2).
  \]
By symmetry,
 \begin{equation} \label{eq-just-awesome}
  H^{\wedge}\cup W^{\wedge} =  (H_i\cup W_i) \cap (H_j\cup W_j)
 \end{equation}
for all distinct $i, j\in \{1, 2, 3\}$.

If $\{i, j, k\} = \{1, 2, 3\}$, then, by \cref{eq-number-0},
 \begin{eqnarray*}
  H^{\vee} \cup W^{\vee} & = & [(H_i\cup W_i)\setminus (H_j\cup W_j)] \cup [(H_i\cup W_i)\cap (H_j\cup W_j)] \\ & &
  \cup [(H_j\cup W_j)\setminus (H_i\cup W_i)] \\ & = &
  [(H_i\cup W_i)\setminus (H_k\cup W_k)] \cup [(H_i\cup W_i)\cap (H_k\cup W_k)] \\ & &
  \cup [(H_k\cup W_k)\setminus (H_i\cup W_i)].
 \end{eqnarray*}
From \cref{eq-number-1} and \cref{eq-just-awesome} it follows that
 \[
 (H_i\cup W_i)\setminus (H_j\cup W_j) = (H_i\cup W_i)\setminus (H_k\cup W_k),
 \] 
and so $H_j \cup W_j = H_k \cup W_k$. Thus, by \cref{lemma-union}, $(H_j, W_j, \varnothing) = (H_k, W_k, \varnothing)$, producing a contradiction. Hence the diamond $\mathfrak{M_3}$ is not a sublattice of $L(G(E))$, as required.
 
\medskip

\noindent(iii) $\Longrightarrow$ (i): Suppose that $L(G(E))$ is distributive. Since, as mentioned above, every distributive lattice is modular, $L(G(E))$ is modular, and hence also lower-semimodular.
 
\medskip

Now, let us assume that $E$ is simple, and prove that (i) $\iff$ (iv).
 
\medskip

\noindent (i) $\iff$ (iv): It suffices, by \cref{thm-semimod}, to show that there exists a forked vertex in $E$ if and only if there are $e, f\in E^1$ such that $\s(e) = \s(f)$, $\rr(e)\not\geq \rr(f)$, and $\rr(f)\not\geq \rr(e)$.

For the forward direction, suppose that $v$ is a forked vertex in $E$. Then there exist distinct $e, f\in \s^{-1}(v)$ such that $\rr(g)\not \geq \rr(e)$ for all $g\in \s^{-1}(v)\setminus \{e\}$ and $\rr(g) \not \geq \rr(f)$ for all $g\in \s^{-1}(v)\setminus \{f\}$. In particular, $\s(e) = \s(f)$, $\rr(e) \not\geq \rr(f)$, and $\rr(f)\not\geq \rr(e)$.

For the converse, suppose that there exist $v\in E^0$ and $e, f\in \s^{-1}(v)$, such that $\rr(e)\not\geq \rr(f)$ and $\rr(f)\not\geq \rr(e)$. For convenience, we will refer to this situation as $v$ \textit{splitting at $e$ and $f$}. Since $E$ is finite and acyclic, we may assume that $v$ is $\leq$-minimal, among vertices that spilt (i.e., no $u \in E^0$ satisfying $v > u$ splits). Next, using the fact that $E$ is finite and has no parallel edges, among the $g \in \s^{-1}(v)$ satisfying $\rr(g) \geq \rr(e)$ we can find one, denoted $e'$, which is $\leq$-maximal. That is, $\rr(g) \not \geq \rr(e')$ for all $g\in \s^{-1}(v) \setminus \{e'\}$. Likewise, among the $g \in \s^{-1}(v)$ satisfying $\rr(g) \geq \rr(f)$ we can find one, denoted $f'$, such that $\rr(g) \not \geq \rr(f')$ for all $g\in \s^{-1}(v) \setminus \{f'\}$. To conclude that $v$ is forked it suffices to show that $e'\neq f'$.
 
Suppose that $e'=f'$. Then $\rr(e') \geq \rr(e)$ and $\rr(f') \geq \rr(f)$ imply that $e'\neq e$ and $e' \neq f$. Since $E$ has no parallel edges, there must exist $g,h \in E^1$ such that $\s(g) = \rr(e') = \s(h)$, $\rr(g) = \rr(e)$, and $\rr(h) = \rr(f)$. But then $u = \rr(e')=\rr(f')$ splits at $g$ and $h$, and satisfies $v > u$, contrary to the choice of $v$. Thus $e'\neq f'$, as desired.
\end{proof}

\section{Atoms and atomistic congruence lattices}

In this section, we describe the atoms in the congruence lattice of a graph inverse semigroup, and prove \cref{thm-atomistic} and \cref{cor-atomistic}.

For convenience, given appropriate $H, W\subseteq E^0$, by $(H,W, 1_H)$ or $(H,W, 1)$ we denote the Wang triple on $E$ with the trivial cycle function, relative to $H$. That is, $1_H: C(E^0) \to \mathbb{Z}^+ \cup \{\infty\}$ is defined by $1_H(c) = 1$ for all $c \in C(H)$, and $1_H(c) = \infty$ for all $c \in C(E^0 \setminus H)$.

\begin{prop} \label{atoms}
Let $E$ be a graph, and let $(H,W,f)$ be a Wang triple on $E$. Then $(H,W,f)$ is an atom in the lattice of Wang triples on $E$ if and only if one of the following holds:
 \begin{enumerate}[\rm (i)]
  \item $H = \varnothing$, $|W|=1$, and $f = 1_{\varnothing}$;
  \item $H$ is a strongly connected component of $E$, $W = \varnothing$, $|\s^{-1}(v)| \geq 2$ for each $v \in H$ that is not a sink, and $f = 1_H$.
 \end{enumerate}
\end{prop}
\begin{proof}
It is easy to see that the least element in the lattice of Wang triples on $E$ is $(\varnothing, \varnothing, 1_{\varnothing})$.

If $(H,W,f)$ satisfies (i), then clearly $(\varnothing, \varnothing, 1_{\varnothing}) \prec (H,W,f)$, by \cref{min-triple}, and so $(H,W,f)$ is an atom. Now suppose that $(H,W,f)$ satisfies (ii). Then $\{v \in H \mid |\s^{-1}(v)|=1\} = \varnothing$, and there are no hereditary sets $H'$ satisfying $\varnothing \subsetneq H' \subsetneq H$. Thus $(\varnothing, \varnothing, 1_{\varnothing}) \prec (H,W,f)$, by \cref{min-triple}, once again.

Conversely, suppose that $(H,W,f)$ is an atom, i.e., $(\varnothing, \varnothing, 1_{\varnothing}) \prec (H,W,f)$. By \cref{min-triple}, there are three possible cases, which we examine individually.

\textit{Case 1:} $H = \varnothing$, $W = \varnothing$, and $f \prec 1_{\varnothing}$. The last clause cannot be satisfied by any $f$, and so this case is not actually possible.

\textit{Case 2:} $H = \varnothing$, $|W \setminus \varnothing | = 1$, and $f = 1_{\varnothing}$. This is condition (i) above.

\textit{Case 3:} $H \neq \varnothing$, $W = \varnothing$, $\{v \in H \mid |\s^{-1}(v)|=1\} = \varnothing,$ and there are no non-empty hereditary sets $H' \subsetneq H$. The last clause amounts to saying that $\{u \in E^0 \mid v\geq u\} = H$ for all $v \in H$, which implies that $H$ is strongly connected. Finally, given that $W = \varnothing$, it must be the case that $f(c) = 1_H$. Thus condition (ii) is satisfied.
\end{proof}

\begin{cor} \label{str-con}
Let $E$ be a graph, let $(H,W,f)$ be a Wang triple on $E$, and suppose that $(H,W,f)$ is the join of a (possibly infinite) collection of atoms in the lattice of Wang triples on $E$. Then one of the following conditions holds for each $v \in H$:
 \begin{enumerate}[\rm (i)]
  \item $|\s^{-1}(v)| = 0$;
  \item $|\s^{-1}(v)| = 1$, $v$ does not belong to a cycle, and $v > u$ for some $u \in E^0$ such that $|\s^{-1}(u)| \neq 1$;
  \item $|\s^{-1}(v)| \geq 2$, and $\rr(e) \geq v$ for all $e \in \s^{-1}(v)$.
 \end{enumerate}
Moreover, if $(H,W,f)$ is the join of finitely many atoms, then $H$ has finitely many strongly connected components and finitely many $v \in H$ such that $|\s^{-1}(v)| = 1$.
\end{cor}

\begin{proof}
Write
\begin{equation} \label{atom-join-1}
   (H,W,f) = \bigvee_{i \in I}(\varnothing,W_i,1) \vee  \bigvee_{i \in K}(H_i,\varnothing,1),
\end{equation}
where each $(\varnothing,W_i,1)$ is of the form described in (i) of \cref{atoms}, and each $(H_i,\varnothing,1)$ is of the form described in (ii) of \cref{atoms}. 
 
We claim that $\bigvee_{i \in I}(\varnothing,W_i,1) = (\varnothing,\bigcup_{i \in I}W_i,1)$. (Note that, since this is a potentially infinite join, \cref{join-meet} does not apply.) Clearly $(\varnothing,W_i,1) \subseteq (\varnothing,\bigcup_{i \in I}W_i,1)$ for each $i \in I$. Now suppose that each $(\varnothing,W_i,1) \subseteq (H',W',f')$ for some Wang triple $(H',W',f')$. Then $\bigcup_{i \in I}W_i \subseteq H' \cup W'$, and so $(\varnothing,\bigcup_{i \in I}W_i,1) \subseteq (H',W',f')$. It follows that $\bigvee_{i \in I}(\varnothing,W_i,1) = (\varnothing,\bigcup_{i \in I}W_i,1)$. An even simpler argument shows that $\bigvee_{i \in K}(H_i,\varnothing,1) = (\bigcup_{i \in K} H_i,\varnothing,1)$.
 
By the definition of Wang triples and \cref{atoms}, each $W_i$ consists of vertices $v$ satisfying $|\s^{-1}(v)| = 1$, while each $H_i$ consists of vertices $v$ satisfying $|\s^{-1}(v)| \neq 1$, and so $(\bigcup_{i \in I}W_i) \cap (\bigcup_{i \in K} H_i) = \varnothing$. Thus, by \cref{join-meet},
 \begin{equation} \label{atom-join-2}
   (H,W,f) = \Big(\varnothing,\bigcup_{i \in I}W_i,1\Big) \vee \Big(\bigcup_{i \in K} H_i,\varnothing,1\Big) = \Big(J \cup \bigcup_{i \in K} H_i,\Big(\bigcup_{i \in I}W_i\Big) \setminus J,1\Big),
 \end{equation}
where
 \[
   V_0 = \Big\{u \in \bigcup_{i \in I}W_i \ \Big| \ \s_{E\setminus (\bigcup_{i \in K} H_i)}^{-1}(u) = \varnothing\Big\}
 \] and
\begin{eqnarray*}
 J = \Big\{u \in \bigcup_{i \in I}W_i \ \Big| \ \exists e_1\cdots e_n \in \Path(E) \ \forall i\in \{2, \dots, n\}\\ \Big(\s(e_1) = u, \, \rr(e_n) \in V_0, \, \s(e_i) \in \bigcup_{i \in I}W_i\Big)\Big\}.
\end{eqnarray*}

Now, let $v \in H$, and suppose that $|\s^{-1}(v)| \neq 0$. If $|\s^{-1}(v)| \geq 2$, then $v \in \bigcup_{i \in K} H_i$, since $J \subseteq \bigcup_{i \in I}W_i$ and $H = J \cup \bigcup_{i \in K} H_i$. Thus $v \in H_i$ for some $i \in K$, where $H_i$ is hereditary and strongly connected, by \cref{atoms}. It follows that $\rr(e) \geq v$ for all $e \in \s^{-1}(v)$, and hence (iii) is satisfied.

Next, suppose that $|\s^{-1}(v)| = 1$. Then, by similar reasoning, $v \in J$, and so, in particular, $v$ does not belong to a cycle. Moreover $v > u$ for some $u \in H_i$ and $i \in K$. By \cref{atoms}, either $|\s^{-1}(u)| = 0$ or $|\s^{-1}(u)| \geq 2$, and so (ii) holds.
 
For the final statement, suppose that $(H,W,f)$ satisfies \cref{atom-join-1}, and therefore also \cref{atom-join-2}, with $I$ and $K$ finite. By construction, the strongly connected components in $H$ are precisely the $H_i$, and the $v \in H$ satisfying $|\s^{-1}(v)| = 1$ are elements of the singleton sets $W_i$. It follows that $H$ can contain only finitely many strongly connected components and $v \in H$ such that $|\s^{-1}(v)| = 1$.
\end{proof}

We can now give the proof of \cref{thm-atomistic}.

\begin{proof}[Proof of \cref{thm-atomistic}.]
($\Longrightarrow$): If every congruence on $G(E)$ is a join of atoms, then conditions (i), (ii), and (iii) hold for every $v \in E^0$, by \cref{str-con}, applied to $(E^0, \varnothing, 1)$. Moreover, if every congruence on $G(E)$ is the join of finitely many atoms (i.e., $L(G(E))$ is atomistic), then, again applying \cref{str-con} to $(E^0, \varnothing, 1)$, shows that $E^0$ has only finitely many strongly connected components and vertices $v$ such that $|s^{-1}(v)| = 1$.
  
\medskip

\noindent ($\Longleftarrow$): Suppose that each $v\in E^0$ satisfies one of the conditions (i), (ii), and (iii) in the statement of the theorem, and that $(H,W,f)$ is a Wang triple on $E$. Then $W$ consists entirely of vertices $v \in E^0$ such that $|\s^{-1}(v)| = 1$ and $v$ does not belong to a cycle. Thus $C(W) = \varnothing$, and so necessarily $f = 1_H = 1$. Moreover, $(\varnothing, W, 1)$ and $(\varnothing,\{v\},1)$ are well-defined Wang triples, for all $v \in W$. It is easy to see (as in the proof of \cref{str-con}) that $(\varnothing, W, 1) = \bigvee_{v \in W}(\varnothing,\{v\},1)$, which is a join of atoms, by \cref{atoms}.

Next, again by hypothesis, we can write 
   \begin{equation} \label{break-down-eq}
    H = U \cup \bigcup_{i\in I} H_i \cup \bigcup_{i\in K} \{w_i\},
   \end{equation}
where each $w_i$ is a sink, each $H_i$ is hereditary and strongly connected, with $|\s^{-1}(v)| \geq 2$ for all $v \in H$, and $U$ consists of vertices satisfying condition (ii). Then $\bigcup_{i\in I} H_i$ and $\bigcup_{i\in K} \{w_i\}$ are hereditary. Also, once again, it is easy to see that $(\varnothing, U, 1) = \bigvee_{v \in U} (\varnothing, \{v\}, 1)$, $(\bigcup_{i\in I} H_i, \varnothing, 1) = \bigvee_{i\in I}(H_i, \varnothing, 1)$, and $(\bigcup_{i\in K} \{w_i\}, \varnothing, 1) = \bigvee_{i\in K} (\{w_i\}, \varnothing, 1)$, where all the Wang triples involved are well-defined. Moreover, by condition (ii), for each $v \in U$, either $v > u$ for some $i \in I$ and $u \in H_i$, or $v > w_i$ for some $i \in K$. It follows, using \cref{join-meet}, that
  \begin{equation*}
    (H, \varnothing, 1) = (\varnothing, U, 1) \vee \Big(\bigcup_{i\in I} H_i, \varnothing, 1\Big) \vee \Big(\bigcup_{i\in K} \{w_i\}, \varnothing, 1\Big),
  \end{equation*}
and so  
  \begin{equation} \label{break-down-eq2} 
    (H, \varnothing, 1) = \bigvee_{v \in U} (\varnothing, \{v\}, 1) \vee \bigvee_{i\in I}(H_i, \varnothing, 1) \vee \bigvee_{i\in K} (\{w_i\}, \varnothing, 1), 
  \end{equation} 
which is a join of atoms, by \cref{atoms}. Noting that, by \cref{join-meet},
  \begin{equation} \label{break-down-eq3} 
    (H,W,f) = (H,W, 1) = (H, \varnothing, 1) \vee (\varnothing, W, 1) = (H, \varnothing, 1) \vee \bigvee_{v \in W}(\varnothing,\{v\},1),
  \end{equation} 
we conclude that $(H,W,f)$ is a join of atoms.
  
Finally, if $E^0$ has only finitely many strongly connected components and vertices $v$ such that $|s^{-1}(v)| = 1$, then $W$, as well as $U$, $I$, and $K$ in \cref{break-down-eq}, must be finite. Hence, $(H,W,f)$ is a finite join of atoms, by \cref{break-down-eq2} and \cref{break-down-eq3}.
\end{proof}

\begin{proof}[Proof of \cref{cor-atomistic}]
(i) $\Longrightarrow$ (ii): Suppose that $|\s^{-1}(v)| \leq 1$ for all $v\in E^0$. We define a mapping $\Psi: L(G(E)) \to \mathcal{P}(E^0)$ as follows:
 \[
  \Psi((H, W, \varnothing)) = H\cup W,
 \]
for every Wang triple $(H, W, \varnothing)$ on $E$. We will show that $\Psi$ is a lattice isomorphism.

By \cref{lemma-union}, $\Psi$ is injective. To show that is it surjective, let $V \in \mathcal{P}(E^0)$. By (i), we can write $V = U \cup W$, where $U$ consists of sinks and $W$ consists of vertices $v$ satisfying $|\s^{-1}(v)| = 1$ (with either set possibly empty). Then $(\{v\}, \varnothing, \varnothing)$ is a Wang triple for each $v \in U$, and $(\varnothing, \{v\}, \varnothing)$ is a Wang triple for each $v \in W$. By \cref{join-meet},
 \[
  \bigvee_{v \in U}(\{v\}, \varnothing, \varnothing)\vee \bigvee_{v \in W}(\varnothing, \{v\}, \varnothing) = (H', W', \varnothing)
 \]
for some sets $H'$ and $W'$ such that $V = H'\cup W'$, and so $\Psi((H', W', \varnothing)) = V$, from which it follows that $\Psi$ is surjective. It remains to show that $\Psi$ preserves meets and joins.

Let $(H_1, W_1, \varnothing)$ and $(H_2, W_2, \varnothing)$ be Wang triples on $E$, and let $J$ be the set from \cref{join-meet}. Then
 \begin{eqnarray*}
  \Psi((H_1, W_1, \varnothing) \vee (H_2, W_2, \varnothing))
  & = & \Psi((H_1\cup H_2\cup J, (W_1\cup W_2)\setminus (H_1\cup H_2\cup J), \varnothing)) \\
  & = & (H_1\cup H_2\cup J) \cup ((W_1\cup W_2)\setminus (H_1\cup H_2\cup J)) \\
  & = & (H_1\cup W_1) \cup (H_2\cup W_2) \\
  & = & \Psi((H_1, W_1, \varnothing)) \cup \Psi((H_2, W_2, \varnothing)).
 \end{eqnarray*}
Next, since $|\s ^ {-1}(v)|\leq 1$ for all $v\in E ^ 0$, the graph $E$ cannot have forked vertices. Thus, by \cref{no-fork},
 \[
  (H_1, W_1, \varnothing) \wedge (H_2, W_2, \varnothing) = (H_1\cap H_2, (W_1\cap H_2)\cup (W_2 \cap H_1) \cup (W_1\cap W_2), \varnothing),
\]
and so
 \begin{eqnarray*}
  \Psi((H_1, W_1, \varnothing) \wedge (H_2, W_2, \varnothing))
  & = &
  \Psi((H_1\cap H_2, (W_1\cap H_2)\cup (W_2 \cap H_1) \cup (W_1\cap W_2), \varnothing))\\
  & = & (H_1\cap H_2)\cup (W_1\cap H_2)\cup (W_2 \cap H_1) \cup (W_1\cap W_2) \\
  & = & (H_1\cup W_1) \cap (H_2\cup W_2) \\
  & = & \Psi((H_1, W_1,\varnothing)) \cap \Psi((H_2, W_2, \varnothing)).
 \end{eqnarray*}

\noindent (ii) $\Longrightarrow$ (iii):
The power set lattice $\mathcal{P}(E^0)$ is, by definition, atomistic.
 
\medskip

\noindent (iii) $\Longrightarrow$ (i):
Suppose that $L(G(E))$ is atomistic. Then one of the conditions (i), (ii), or (iii) in \cref{thm-atomistic} holds for every $v\in E^0$. Since $E$ is assumed to be acyclic, no $v \in E^0$ satisfies condition (iii) in that theorem. Hence $|\s^{-1} (v)| \leq 1$ for every $v\in E^0$.
\end{proof}

\section{Generating congruences}

The purpose of this section is to prove \cref{thm-generators}.

\begin{proof}[Proof of \cref{thm-generators}]
$(\Longleftarrow)$: Suppose that $\mathcal{A}$ contains all the congruences of types (i) and (ii) in the statement of the theorem, and let $(H, W, \varnothing)$ be any Wang triple on $E$. Since $E$ is finite and acyclic, we can write $H = \bigcup_{i=0}^n H_i$, for some $n\geq 0$, where for each $i$, the set $H_i$ consists of the vertices $h \in H$ such that there exists $p \in \Path(E)$ of length $i$, with $\s(p)=h$ and $\rr(p)$ a sink, and such that $i$ is maximal for this property of $h$. For each $i \geq 1$ and $h \in H_i$ choose $e_h \in  \s^{-1}(h)$ such that $\rr(e_h) \in H_{i-1}$, and let 
\[
G_h = \{v\in E^0 \mid \rr(f) \geq v \text{ for some } f \in \s^{-1}(h)\setminus\{e_h\}\}.
\] 
By the construction of $H_i$ and the hypothesis that $E$ is simple, here $\rr(f) \not \geq \rr(e_h)$ for all $f\in \s^{-1}(h) \setminus \{e_h\}$, and so, in particular, $\rr(e_h) \notin G_h$. Notice also that $(\{h\}, \varnothing, \varnothing)$ is a well-defined Wang triple for each $h \in H_0$, that $(G_{h}, \{h\}, \varnothing)$ is a well-defined Wang triple for each $h \in H_i$ with $i \geq 1$, and that both belong to $\mathcal{A}$. We will next show, by induction on $n$, that
\[
(H, \varnothing, \varnothing) = \bigvee_{h \in H_0}(\{h\}, \varnothing, \varnothing) \vee \bigvee_{i=1}^n \bigvee_{h \in H_i} (G_{h}, \{h\}, \varnothing).
\]

If $n=0$, and so $H = H_0$ consists of sinks, then it follows immediately from \cref{join-meet} that $(H, \varnothing, \varnothing) = \bigvee_{h \in H_0}(\{h\}, \varnothing, \varnothing)$. Supposing that $n \geq 1$, let us assume inductively that
\[
  \Big(\bigcup_{i=0}^{n-1} H_i, \varnothing, \varnothing\Big) = \bigvee_{h \in H_0}(\{h\}, \varnothing, \varnothing) \vee \bigvee_{i=1}^{n-1} \bigvee_{h \in H_i} (G_{h}, \{h\}, \varnothing).
\]
(It is easy to see that $\bigcup_{i=0}^{n-1} H_i$ is hereditary.) By construction, $\s^{-1}_{E\setminus (G_h \cup  H_{n-1})}(h) = \varnothing$ for each $h \in H_n$, and so, by \cref{join-meet},
  \[
    (G_{h}, \{h\}, \varnothing) \vee \Big(\bigcup_{i=0}^{n-1} H_i, \varnothing, \varnothing\Big) = \Big(G_{h} \cup \{h\} \cup \bigcup_{i=0}^{n-1} H_i, \varnothing, \varnothing\Big).
  \]
Since $H = \bigcup_{i=0}^n H_i$, iterating this computation gives
  \[
    \Big(\bigcup_{i=0}^{n-1} H_i, \varnothing, \varnothing\Big) \vee \bigvee_{h \in H_n} (G_{h}, \{h\}, \varnothing) = (H, \varnothing, \varnothing),
  \]
which proves the claim. In particular, if $W = \varnothing$, then $(H, W, \varnothing)$ is a join of congruences from $\mathcal{A}$.

Now suppose that $W \neq \varnothing$, and for each $w \in W$ define $K_w$ to be a minimal hereditary subset of $H$, such that $|\s_{E\setminus K_w}^{-1}(w)| = 1$. Then clearly $(K_w, \{w\}, \varnothing)$ is a Wang triple belonging to $\mathcal{A}$. Since $|\s^{-1}_{E\setminus H}(w)| = 1$, applying \cref{join-meet} once more, gives
\[
  (H, \varnothing, \varnothing) \vee (K_w, \{w\}, \varnothing) = (H, \{w\}, \varnothing)
\]
for each $w \in W$. It follows that
\begin{eqnarray*}
(H, W, \varnothing) & = & (H, \varnothing, \varnothing) \vee \bigvee_{w \in W} (K_w, \{w\}, \varnothing)\\ & = & \bigvee_{h \in H_0}(\{h\}, \varnothing, \varnothing) \vee \bigvee_{i=1}^n \bigvee_{h \in H_i} (G_{h}, \{h\}, \varnothing) \vee \bigvee_{w \in W} (K_w, \{w\}, \varnothing).
\end{eqnarray*}
Thus $\mathcal{A}$ generates $L(G(E))$.
 
\medskip

\noindent $(\Longrightarrow)$: It suffices to show that every congruence $\rho$ of type (i) or (ii) in the statement of the theorem is indecomposable, in the sense that if $\rho = \sigma \vee \tau$, then $\sigma = \rho$ or $\tau = \rho$. Any congruence of type (i) is an atom, by \cref{atoms}. As such, any congruence of this sort is not the join of two or more distinct congruences, and therefore the claim trivially holds for every congruence of type (i).

Now let $(H, \{v\}, \varnothing)$ be a congruence of type (ii). Clearly $(H, \varnothing, \varnothing) \subset (H, \{v\}, \varnothing)$. So to show that $(H, \{v\}, \varnothing)$ is indecomposable it suffices to prove that if $\tau$ is any other congruence such that $\tau \subseteq (H, \{v\}, \varnothing)$, then $\tau \subseteq (H, \varnothing, \varnothing)$ or $\tau = (H, \{v\}, \varnothing)$. If $\tau = (H', W', \varnothing)$ for some $H'$ and $W'$, then $H'\subseteq H$ and $W'\setminus H \subseteq \{v\}$. Hence either $W'\setminus H = \varnothing$ or $W'\setminus H = \{v\}$. In the first case, $W'\subseteq H$, and so $\tau = (H', W', \varnothing) \subseteq (H, \varnothing, \varnothing)$. In the second case, $W'\setminus H = \{v\}$, the condition that $|\s^{-1}_{E\setminus H'}(v)| = 1$ and the minimality of $H$ imply that $H' = H$. Then 
\[
W' = W'\setminus H' = W'\setminus H = \{v\},
\] 
giving $\tau = (H, \{v\}, \varnothing)$, as required.
\end{proof}

\subsection*{Acknowledgement}  

We are grateful to the referee for a very carful reading of the manuscript.

\printbibliography

\medskip

\noindent M.\ Anagnostopoulou-Merkouri, Mathematical Institute, North Haugh, St Andrews, Fife, KY16 9SS, Scotland

\noindent \emph{Email:} \href{mailto:mam49@st-andrews.ac.uk}{mam49@st-andrews.ac.uk}

\medskip

\noindent Z.\ Mesyan, Department of Mathematics, University of Colorado, Colorado Springs, CO, 80918, USA 

\noindent \emph{Email:} \href{mailto:zmesyan@uccs.edu}{zmesyan@uccs.edu}

\medskip

\noindent J.\ D.\ Mitchell, Mathematical Institute, North Haugh, St Andrews, Fife, KY16 9SS, Scotland

\noindent \emph{Email:} \href{mailto:jdm3@st-and.ac.uk}{jdm3@st-and.ac.uk}

\end{document}